\documentclass[reqno]{amsart}
\usepackage{amsmath}
\usepackage{paralist}
\usepackage{graphics} %% add this and next lines if pictures should be in esp format
\usepackage{epsfig} %For pictures: screened artwork should be set up with an 85 or 100 line screen
\usepackage{graphicx}  \usepackage{epstopdf}%This is to transfer .eps figure to .pdf figure; please compile your paper using PDFLeTex or PDFTeXify.
\usepackage{amssymb}
\usepackage{amsthm}
\usepackage{epstopdf}
\usepackage{verbatim}
\usepackage{mathrsfs}
\usepackage{mathtools}
\usepackage{pstricks}
\usepackage{cleveref}
\usepackage{relsize}
\usepackage{tikz}
\usetikzlibrary{matrix}
\usepackage{subcaption}
\usepackage{pgfplots}
\usepackage{fixltx2e}
\usepackage{bm}
\usepackage{appendix}

\newtheorem{theorem}{Theorem}[section]
\newtheorem{corollary}{Corollary}

\newtheorem{lemma}[theorem]{Lemma}
\newtheorem{proposition}{Proposition}

\theoremstyle{definition}
\newtheorem{definition}[theorem]{Definition}

\newtheorem{example}{Example}

\title[The turnpike with lack of observability] %Use the shortened version of the full title
{The turnpike with lack of observability}

\author[Dario Pighin]{}
\author[Noboru Sakamoto]{}

\email{dario.pighin@uam.es}
\email{noboru.sakamoto@nanzan-u.ac.jp}

\thanks{(Noboru Sakamoto) Supported, in part, by JSPS KAK-
	ENHI Grant Number JP19K04446 and by Nanzan University Pache Research Subsidy I-A-2 for 2019 academic year.\\
	(Dario Pighin)
	This project has received funding from the European Research Council (ERC) under the European Union’s Horizon 2020 research and innovation programme (grant agreement NO. 694126-DyCon).}

\begin{document}
	\maketitle
	
	% Enter the first author's name and address:
	\centerline{\scshape Dario Pighin}
	\medskip
	{\footnotesize
		% please put the address of the first author
		\centerline{Departamento de Matem\'aticas, Universidad Aut\'onoma de Madrid}
		%   \centerline{Other lines}
		\centerline{28049 Madrid, Spain}
	} % Do not forget to end the {\footnotesize by the sign }
	\medskip
	{\footnotesize
		% please put the address of the first author
		\centerline{Chair of Computational Mathematics, Fundaci\'on Deusto}
		%   \centerline{Other lines}
		\centerline{University of Deusto, 48007, Bilbao, Basque Country, Spain}
	} % Do not forget to end the {\footnotesize by the sign }

	\medskip
	
	\centerline{\scshape Noboru Sakamoto}
	\medskip
	{\footnotesize
		% please put the address of the first author
		\centerline{Faculty of Science and Engineering, Nanzan University}
		%   \centerline{Other lines}
		\centerline{Yamazato-cho 18, Showa-ku, Nagoya, 464-8673, Japan}
	} % Do not forget to end the {\footnotesize by the sign }

	%The abstract of your paper
\begin{abstract}
	We obtain turnpike results for optimal control problems with lack of stabilizability in the state equation and/or detectability in the state term in the cost functional.\\
	We show how, under weakened stabilizability/detectability conditions, terminal conditions may affect turnpike phenomena.\\
	Numerical simulations have been performed to illustrate the theoretical results.
	
\end{abstract}

%The title of your section 1
\section*{Introduction}

The purpose of this manuscript is to check the validity of the turnpike property for linear quadratic optimal control problems, with weak observation of the state in the cost functional. We consider the time-evolution optimal control problem
\begin{equation*}
\min_{u}J^{T}(u)=\frac12 \int_{0}^T \left[\|u(t)\|^2+\|Cx(t)-z\|^2\right] dt,
\end{equation*}
where:
\begin{equation*}
\begin{dcases}
\dot{x}=Ax+Bu\hspace{2.8 cm} & \mbox{in} \hspace{0.10 cm}(0,T)\\
\mbox{terminal conditions}
\end{dcases}
\end{equation*}
and the corresponding steady one
\begin{equation*}
\min_{(x,u)}J_s(x,u)=\frac12 \|u\|^2+\frac{1}{2}\|Cx-z\|^2,\hspace{0.6 cm}\mbox{with the constraint} \hspace{0.6 cm}0=Ax+Bu.
\end{equation*}
The time-evolution problem satisfies the \textit{turnpike property} if the time-evolution optimal pair $(u^T,x^T)$ approximate the steady optimal pair $(\overline{u},\overline{x})$ as the time horizon $T\to +\infty$.

Typically, for turnpike to hold, the pair $(A,B)$ is required to be stabilizable and the pair $(A,C)$ is asked to be detectable (see, e.g. \cite{porretta2013long,trelat2015turnpike,TGS}). Our goal is to check if turnpike holds for the full control and the detected state, under weakened detectability and controllability assumptions.

The study of the behaviour of control problems in long time and the turnpike property is a classical topic in the literature. We provide just some essential references. A pioneer on the topic has been the econometrician Paul Samuelson (see \cite{Samuelson1,Samuelson1972,LS}). Later on the topic has been studied both in Mathematics and in Economic Sciences \cite{mckenzie1963turnpike,wilde1972dichotomy,rockafellar1973saddle,haurie1976optimal,anderson1987optimal,rapaport2004turnpike,rapaport2005competition,faulwasser2019towards}. The infinite dimensional case has been explored \cite{allaire2010long,carlson2012infinite,damm2014exponential}. An extensive review on the topic is \cite{zaslavski2006turnpike}. More recently, the topic has been studied in \cite{porretta2013long,trelat2015turnpike,PZ2,trelat2018steady,trelat2018integral,gruneexponential,grune2019sensitivity,pighin2020turnpike}. Related results have been obtained in Mean Field Games (see, for instance, \cite{cardaliaguet2012long,cardaliaguet2013long}).

We distinguish two cases:
\begin{itemize}
	\item section \ref{sec:1}: free endpoint (the state left free at final time $t=T$);
	\item section \ref{sec:2}: fixed endpoint (the state has to match a given final target at time $t=T$).
\end{itemize}

In section \ref{sec:1}, we suppose the state is left free at $t=T$. We can then employ Kalman decomposition (see, e.g. \cite[section 3.3]{zhou1996robust}) to decompose the state space into a detectable part and an undetectable one. In Proposition \ref{prop_eqC}, we prove that an exponential turnpike property is satisfied by the full control and the detected state if and only if the observable modes are stabilizable. In particular, if the state equation is stabilizable, the turnpike property holds for the full control and the detected state, without any observability assumptions on the cost functional (Corollary \ref{cor_ocp}).

These results rely on the absence of final condition for the state. As a consequence of that, unobservable modes do not influence the value of the cost functional, i.e. they are irrelevant for the sake of optimization.

In subsections \ref{subsec:1.2} and \ref{subsec:1.3}, the above results are employed in the context of pointwise control of respectively the heat and wave equation. We project the state equation onto a finite number of Fourier modes. For the heat equation with potential, the full control and the observed state fulfils turnpike if and only if whenever an eigenfunction is non-zero on the observation point, either the same eigenfunction is non-zero on the control point or the point is stable for the free dynamics. For the wave equation, a more restricted condition is required. The turnpike property is verified by the full control and the observed state if and only if whenever an eigenfunction vanishes on the control point, it vanishes on the observation point as well.

In section \ref{sec:2}, we deal with a fixed endpoint problem. Therefore,
\begin{itemize}
	\item on the one hand, in the running cost we penalize only the observed state;
	\item on the other hand, the unobservable modes are relevant to fulfill the final condition. Namely, the unobservable component of the system enters in the definition of the set of admissible controls, where the functional is minimized.
\end{itemize}
For this reason, we need to assume controllability of the full state equation and we cannot employ Kalman decomposition to get rid of the unobservable component. In Proposition \ref{prop_controllability_case}, we prove that the turnpike property is verified for full control and state, if the unobservable modes of $A$ are not critical (they are not associated to a purely imaginary eigenvalue of the free dynamics). In particular, turnpike can hold even if both stable and unstable modes of $A$ are not observable. This condition is formulated as weak Hautus test.

Inspired by \cite{faulwasser2019towards}, we study a class of optimal control problems with fixed endpoint, where the Hamiltonian matrix may have imaginary eigenvalues. In Proposition \ref{prop_controllability_case_velocity} we prove that basically the exponential turnpike is satisfied by the control and the observed state, while the unobserved state is linear in time, up to an exponentially small remainder. We show how this result applies to the illustrative example in \cite[section 3]{faulwasser2019towards}.

\section{Free endpoint problem}
\label{sec:1}

\subsection{Statement of the main results}
\label{subsec:1.1}

We consider the linear quadratic optimal control problem:
\begin{equation}\label{functional_ocp}
\min_{u\in L^2(0,T;\mathbb{R}^m)}J^{T}(u)=\frac12 \int_{0}^T \left[\|u(t)\|^2+\|Cx(t)-z\|^2\right] dt,
\end{equation}
where:
\begin{equation}\label{linear_ocp}
\begin{dcases}
\dot{x}=Ax+Bu\hspace{2.8 cm} & \mbox{in} \hspace{0.10 cm}(0,T)\\
x(0)=x_0.
\end{dcases}
\end{equation}
The matrix $A\in\mathcal{M}_{n\times n}(\mathbb{R})$ describes the free dynamics, while the action of the control is defined by multiplication by the matrix $B\in\mathcal{M}_{n\times m}(\mathbb{R})$. $C\in\mathcal{M}_{n\times n}(\mathbb{R})$ is an observation matrix. By the Direct Methods in the Calculus of Variations and strict convexity, the above problem admits a unique optimal control denoted by $u^T$. The optimal state is denoted by $x^T$. Furthermore, by strict convexity, the optimal control $u^T$ is the unique solution to the optimality system

\begin{equation}\label{OS}
\begin{dcases}
\dot{x}^T(t)=Ax^T(t)-BB^*p^T(t)\hspace{1 cm}& t\in (0,T)\\
-\dot{p}^T(t)=A^*p^T(t)+C^*(Cx^T(t)-z)&t\in (0,T)\\
u^T(t)=-B^*p^T(t)&t\in (0,T)\\
x^T(0)=x_0\\
p^T(T)=0.
\end{dcases}
\end{equation}

The corresponding steady problem reads as
\begin{equation}\label{steady_functional_linear_quadratic_finite_dimension}
\min_{(x,u)}J_s(x,u)=\frac12 \|u\|^2+\frac{1}{2}\|Cx-z\|,\hspace{0.6 cm}\mbox{with the constraint} \hspace{0.6 cm}0=Ax+Bu.
\end{equation}

The well posedeness of the steady problem follows from the following Lemma.

\begin{lemma}\label{lemma_minimizer steady}
	Let $A\in \mathcal{M}_{n\times n}(\mathbb{R})$, $B\in\mathcal{M}_{n\times m}(\mathbb{R})$, $C\in\mathcal{M}_{n\times n}(\mathbb{R})$ and $z\in\mathbb{R}^n$. Set
	\begin{equation*}
	M\coloneqq\left\{ (u,x)\in\mathbb{R}^m\times\mathbb{R}^n \ | \ 0=Ax+Bu\right\}
	\end{equation*}
	and
	\begin{equation}\label{Js}
	J_s(u,x) \coloneqq \frac12\left[\|u\|^2+\|Cx-z\|^2\right].
	\end{equation}
	Then,
	\begin{enumerate}
		\item there exists $(\overline{u},\overline{x})\in M$ global minimizer for $J_s$ over $M$;
		\item the set of global minimizers of $J_s$ is given by
		\begin{equation*}
		\mbox{argmin}(J_s)=\left\{(\overline{u},\overline{x})\right\}+\left\{0\right\}\times\left[\ker(A)\cap \ker(C)\right] .
		\end{equation*}
	\end{enumerate}
\end{lemma}

We prove this Lemma in the Appendix.

Inspired by \cite[Definition 2.1, page 480]{BRC}, we give the following definition of $C$-stabilizability, namely stabilizability of the observable part of the state.
\begin{definition}\label{Cstabilizability}
	Let $A\in\mathcal{M}_{n\times n}(\mathbb{R})$, $B\in\mathcal{M}_{n\times m}(\mathbb{R})$ and $C\in\mathcal{M}_{n\times n}(\mathbb{R})$. $(A,B)$ is said to be $C$-stabilizable if there exists a feedback matrix $L\in \mathcal{M}_{m\times n}(\mathbb{R})$, such that
	\begin{equation*}
	\|C\exp\left(t\left(A-BL\right)\right)\|\leq K\exp(-\mu t),\hspace{0.3 cm}\forall \ t\geq 0,
	\end{equation*}
	for some $K$, $\mu>0$.
\end{definition}

We introduce the concept of $C$-turnpike, i.e. turnpike for the full control and the detected state. To this end, let us decompose the state space into a detectable part and an undetectable one. We start be defining some observability and detectability concepts. Let $T>0$ be a time horizon. The output operator $\psi_{T}:\mathbb{R}^n\longrightarrow L^2(0,+\infty;\mathbb{R}^n)$ is defined as
\begin{equation}\psi_{T}\left(\varphi_0\right)\coloneqq
\begin{cases}
C\exp\left(A t\right)\varphi_0 \hspace{0.3 cm} &t\in [0,T]\\
0 \hspace{0.3 cm} &t>T\\
\end{cases}
\end{equation}
for any $\varphi_0\in \mathbb{R}^n$. The subspace $NO(C,A)\coloneqq \ker\left(\psi_{T}\right)$ is called the unobservable space, which admits an algebraic representation \cite[Propositon 1.4.7]{OCO}
\begin{equation}\label{alg_repr_unobservable space}
	\ker\left(\psi_{T}\right)=\bigcap_{i=0}^{n-1}\ker\left(CA^i\right).
\end{equation}
We define the observable space as the orthogonal $\ker\left(\psi_{T}\right)^{\perp}$. The undetectbale space is defined as the subspace $NO^{0+}(C,A)\coloneqq \ker\left(\psi_{T}\right)\cap \mathscr{L}^{0+}(A)$, made of those unobservable modes which are not stable. By \eqref{alg_repr_unobservable space},
\begin{equation*}
	NO^{0+}(C,A)=\bigcap_{i=0}^{n-1}\ker\left(CA^i\right)\cap \mathscr{L}^{0+}(A).
\end{equation*}
The detectable space is defined as $W\coloneqq NO^{0+}(C,A)^{\perp}$.

We are now in position to decompose the state space into a detectable part and an undetectable one
\begin{equation*}
	\mathbb{R}^n=W\oplus NO^{0+}(C,A),
\end{equation*}
where
\begin{equation*}
	NO^{0+}(C,A) = \bigcap_{i=0}^{n-1}\ker\left(CA^i\right)\cap \mathscr{L}^{0+}(A)
\end{equation*}
and
\begin{equation*}
	W \coloneqq NO^{0+}(C,A)^{\perp}.
\end{equation*}
The matrix associated to the orthogonal projection onto $W$ is denoted by $D$, while the matrix associated to the projection onto $NO^{0+}(C,A)$ is indicated by $R$. For any $x\in \mathbb{R}^n$,
\begin{equation}\label{def_D}
	x=Dx+Rx.
\end{equation}
If $(A,C)$ is detectable, $D=I$ is the identity matrix.

\begin{definition}\label{Cturnpike}
	Let $A\in\mathcal{M}_{n\times n}(\mathbb{R})$, $B\in\mathcal{M}_{n\times m}(\mathbb{R})$ and $C\in\mathcal{M}_{n\times n}(\mathbb{R})$. Let $D$ be the corresponding projection onto the detectable space $W$, as in \eqref{def_D}. The triplet $(A,B,C)$ enjoys $C$-turnpike if, for any initial datum $x_0\in\mathbb{R}^n$ and target $z\in\mathbb{R}^n$, there exists $K=K(A,B,C,x_0,z)$ and $\mu=\mu(A,B,C)>0$, such that, for any $T>0$,
	\begin{equation*}
	\|u^T(t)-\overline{u}\|+\|Dx^T(t)-D\overline{x}\|\leq K\left[\exp(-\mu t)+\exp\left(-\mu\left(T-t\right)\right)\right],
	\end{equation*}
	where $(u^T,x^T)$ is the optimal pair for \eqref{linear_ocp}-\eqref{functional_ocp} and $\left(\overline{u},\overline{x}\right)$ is a minimizer for the steady functional \eqref{Js}.
\end{definition}

As we announced, the main assumption of the following Proposition is that observable modes are stabilizable.

\begin{proposition}\label{prop_eqC}
	In the above notation, the triplet $(A,B,C)$ enjoys $C$-turnpike if and only if $(A,B)$ is $C$-stabilizable.
\end{proposition}
The proof can be found in subsection \ref{subsec:1.4}.

We have the following Corollary.
\begin{corollary}\label{cor_ocp}
	Suppose
	\begin{equation}\label{Hypothesis_1}
	(A,B)\hspace{0.3 cm}\mbox{is stabilizable.}\tag{H}
	\end{equation}
	Then, for any observation matrix $C\in\mathcal{M}_{n\times n}(\mathbb{R})$, the triplet $(A,B,C)$ enjoys $C$-turnpike.
\end{corollary}
\begin{proof}[Proof of Corollary \ref{cor_ocp}]
	%\textit{Step 1} \  \textbf{Conclusion}\\
	$(A,B)$ is stabilizable. Then, for any $C$, $(A,B)$ is $C$-stabilizable. Hence, Proposition \ref{prop_eqC} yields the conclusion.
\end{proof}

In subsection \ref{subsec:1.2} we apply our theory to the pointwise control of the heat equation and in subsection \ref{subsec:1.3} we illustrate how our theory works in the pointwise control of the wave equation. In subsection \ref{subsec:1.4}, we prove Proposition \ref{prop_eqC}.

\subsection{Pointwise control of the heat equation}
\label{subsec:1.2}
Let $\Omega$ be a connected bounded open set of $\mathbb{R}^n$, $n=1,2,3$, with $C^{\infty}$ boundary. Inspired by \cite[section 1.2]{doi:10.1137/1.9781611970982.ch1}, we consider a heat equation controlled from one point $x_{\mbox{\tiny{con}}}\in \Omega$
\begin{equation}\label{heat_point_1}
\begin{cases}
y_t-\Delta y +cy =v(t)\delta(x-x_{\mbox{\tiny{con}}})\hspace{0.6 cm} & \mbox{in} \hspace{0.10 cm}(0,T)\times \Omega\\
y=0  & \mbox{on}\hspace{0.10 cm} (0,T)\times \partial \Omega\\
y(0,x)=y_0(x),  & \mbox{in} \hspace{0.10 cm}\Omega\\
\end{cases}
\end{equation}
where $y=y(t,x)$ is the state, while $v=v(t)$ is the control and $\delta(x-x_{\mbox{\tiny{con}}})$ is the Dirac delta at $x_{\mbox{\tiny{con}}}$, i.e. the control acts on $x_{\mbox{\tiny{con}}}$. The potential coefficient $c$ is supposed to be bounded. Following \cite[subsection 1.2.2]{doi:10.1137/1.9781611970982.ch1}, one can prove that for any $y_0\in L^2(\Omega)$ and $v\in L^2(0,T)$, there exists a unique $y\in L^2((0,T)\times \Omega)$ solution by transposition to \eqref{heat_point_1}, with initial datum $y_0$ and control $v$.

We derive now a finite-dimensional Fourier approximation of the above controlled equation. We assume that the spectrum of $\mathcal{A}\coloneqq -\Delta+cI:H^1_0(\Omega)\longrightarrow H^{-1}(\Omega)$ is simple. Let $\left\{\lambda_k\right\}$ be the spectrum of $\mathcal{A}$ and let $\left\{\phi_k\right\}$ be a corresponding set of eigenfunctions, orthonormal basis of $L^2(\Omega)$.

Fix $N\in\mathbb{N}\setminus \left\{0\right\}$. The projection of \eqref{heat_point_1} on the finite dimensional space $\mbox{span} \left\{\phi_1,\dots,\phi_N\right\}$ reads as
\begin{equation}\label{linear_ocp_heat_N}
\begin{dcases}
\dot{x}=A_Nx+B_Nu\hspace{2.8 cm} & \mbox{in} \hspace{0.10 cm}(0,T)\\
x(0)=x_0,
\end{dcases}
\end{equation}
where $A_N$ is an $N\times N$ diagonal matrix
\begin{equation}\label{diagonal_N}
A_N=\begin{tikzpicture}[baseline=(current bounding box.center)]
\matrix (A) [matrix of math nodes,nodes in empty cells,right delimiter={]},left delimiter={[} ]{
	-\lambda_1  &   &             &    &    &    &    &             &    &    \\
	&    &             &    &    &    &    & \textbf{0}  &    &    \\
	&    &             &    &    &    &    &             &    &    \\
	&    &             &    &    &    &    &             &    &    \\
	&    &             &    &    &    &    &             &    &    \\
	&    &             &    &    &    &    &             &    &    \\
	&    &             &    &    &    &    &             &    &    \\
	&    &             &    &    &    &    &             &    &    \\
	&    & \textbf{0}  &    &    &    &    &             &    &  \\
	&    &             &    &    &    &    &             &    & -\lambda_N  \\
} ;
\draw[loosely dotted] (A-1-1)--(A-10-10);
\end{tikzpicture},
\end{equation}
the control operator $B_N$ is an $N\times 1$ matrix
\begin{equation}\label{control_oper_N}
B_N=
\begin{tikzpicture}[baseline=(current bounding box.center)]
\matrix (b) [matrix of math nodes,nodes in empty cells,right delimiter={]},left delimiter={[} ]{
	\phi_1(x_{\mbox{\tiny{con}}})  \\
	\phi_2(x_{\mbox{\tiny{con}}})  \\
	\\
	\\
	\\
	\\
	\\
	\\
	\\
	\phi_N(x_{\mbox{\tiny{con}}})  \\
} ;
\draw[loosely dotted] (b-3-1)-- (b-10-1);
\end{tikzpicture}
\end{equation}
and the initial datum
\begin{equation*}
x_0=\left[\int_{\Omega}y_0\phi_1dx,\dots,\int_{\Omega}y_0\phi_Ndx\right].
\end{equation*}

We consider the optimal control problem:
\begin{equation}\label{functional_ocp_ex_heat_N}
\min_{u\in L^2(0,T)}J^{T}(u)=\frac12 \int_{0}^T \left[|u(t)|^2+\|C_Nx(t)-z\|^2\right] dt,
\end{equation}
where $x$ is the state solution to \eqref{linear_ocp_heat_N}, with control $u$ and initial datum $x_0$, the scalar $z\in \mathbb{R}$ is a running target and
\begin{equation*}
C_N\coloneqq \left[\phi_1(x_{\mbox{\tiny{obs}}}),\dots, \phi_N(x_{\mbox{\tiny{obs}}})\right],
\end{equation*}
namely the state is observed on $x_{\mbox{\tiny{obs}}}\in\Omega$.

\begin{proposition}\label{prop_point_heat}
	The triplet $(A_N,B_N,C_N)$ enjoys $C_N$-turnpike if and only if for any $i\in \left\{1,\dots, N\right\}$ such that $\phi_i(x_{\mbox{\tiny{obs}}})\neq 0$, either $\phi_i(x_{\mbox{\tiny{con}}})\neq 0$ or $\lambda_i> 0$.
\end{proposition}
\begin{proof}[Proof of Proposition \ref{prop_point_heat}]
	In this case, the observable subspace reads as
	\begin{equation}\label{def_W_N}
	W_N = NO(C_N,A_N)^{\perp}=\sum_{i=0}^{n-1}\mbox{Range}\left(\left(A_N^*\right)^iC_N^*\right)
	\end{equation}
	and the stabilizable subspace
	\begin{equation*}
	S(A_N,B_N)= \sum_{i=0}^{n-1}\mbox{Range}\left(A_N^iB_N\right)+\mathscr{L}^{-}(A_N).
	\end{equation*}
	$(A_N,B_N)$ is $C_N$-stabilizable if and only if
	\begin{equation*}
	W_N\subseteq S(A_N,B_N).
	\end{equation*}
	Now, for any natural $i\geq 0$, we have
	\begin{equation*}
	\left(A_N^*\right)^iC_N^*=
	\begin{tikzpicture}[baseline=(current bounding box.center)]
	\matrix (b) [matrix of math nodes,nodes in empty cells,right delimiter={]},left delimiter={[} ]{
		\lambda_1^i\phi_1(x_{\mbox{\tiny{obs}}})  \\
	    \\
		\\
		\\
		\\
		\\
		\\
		\\
		\\
		\lambda_N^i\phi_N(x_{\mbox{\tiny{obs}}})  \\
	} ;
	\draw[loosely dotted] (b-1-1)-- (b-10-1);
	\end{tikzpicture}
	\end{equation*}
	and
	\begin{equation*}
\left(A_N\right)^iB_N=
\begin{tikzpicture}[baseline=(current bounding box.center)]
\matrix (b) [matrix of math nodes,nodes in empty cells,right delimiter={]},left delimiter={[} ]{
	\lambda_1^i\phi_1(x_{\mbox{\tiny{con}}})  \\
	\\
	\\
	\\
	\\
	\\
	\\
	\\
	\\
	\lambda_N^i\phi_N(x_{\mbox{\tiny{con}}})  \\
} ;
\draw[loosely dotted] (b-1-1)-- (b-10-1);
\end{tikzpicture}.
\end{equation*}	
	Hence, since the spectrum of the Dirichlet laplacian is simple, the $(C_N,A_N)$ observable subspace is
	\begin{equation*}
	W_N=\left\{w\in\mathbb{R}^N \ | \ w_i=0, \ \mbox{if} \ \phi_i(x_{\mbox{\tiny{obs}}})= 0\right\}
	\end{equation*}
	and the $(A_N,B_N)$ stabilizable subspace reads as
	\begin{equation*}
	S(A_N,B_N)=\left\{w\in\mathbb{R}^N \ | \ w_i=0, \ \mbox{if} \ \phi_i(x_{\mbox{\tiny{con}}})=0 \ \mbox{and} \ \lambda_i\leq 0 \right\}.
	\end{equation*}
	Then, $W_N\subseteq S(A_N,B_N)$ if and only if
	\begin{equation*}
	\left\{i\in \left\{1,\dots,N\right\} \ | \ \phi_i(x_{\mbox{\tiny{obs}}})\neq 0\right\}\subset \left\{i\in \left\{1,\dots,N\right\} \ | \ \phi_i(x_{\mbox{\tiny{con}}})\neq 0 \ \mbox{and} \ \lambda_i> 0\right\}
	\end{equation*}
	
	Then, by Proposition \ref{prop_eqC}, the triplet $(A_N,B_N,C_N)$ enjoys $C_N$-turnpike if and only if for any $i\in \left\{1,\dots, N\right\}$ such that $\phi_i(x_{\mbox{\tiny{obs}}})\neq 0$, either $\phi_i(x_{\mbox{\tiny{con}}})\neq 0$ or $\lambda_i> 0$.
\end{proof}

We have performed some numerical simulations employing a conjugate gradient method (see \cite[algorithm 2 page 111]{ROP}). In our notation, we chose domain $\Omega=(0,10)$, $N=16$, potential coefficient $c\equiv -(2\pi/10)^2-1$, $x_{\mbox{\tiny{con}}} = L/3$, $x_{\mbox{\tiny{obs}}} = L/2$, target $z\equiv 1$ and initial datum $x_0=[1,\dots,1]$. The results are depicted in figures \ref{controlheat} and \ref{observedstateheat}.
\begin{figure}[htp]
	\begin{center}
		% replace aims_logo.eps by your figure file name
		\includegraphics[width=12cm]{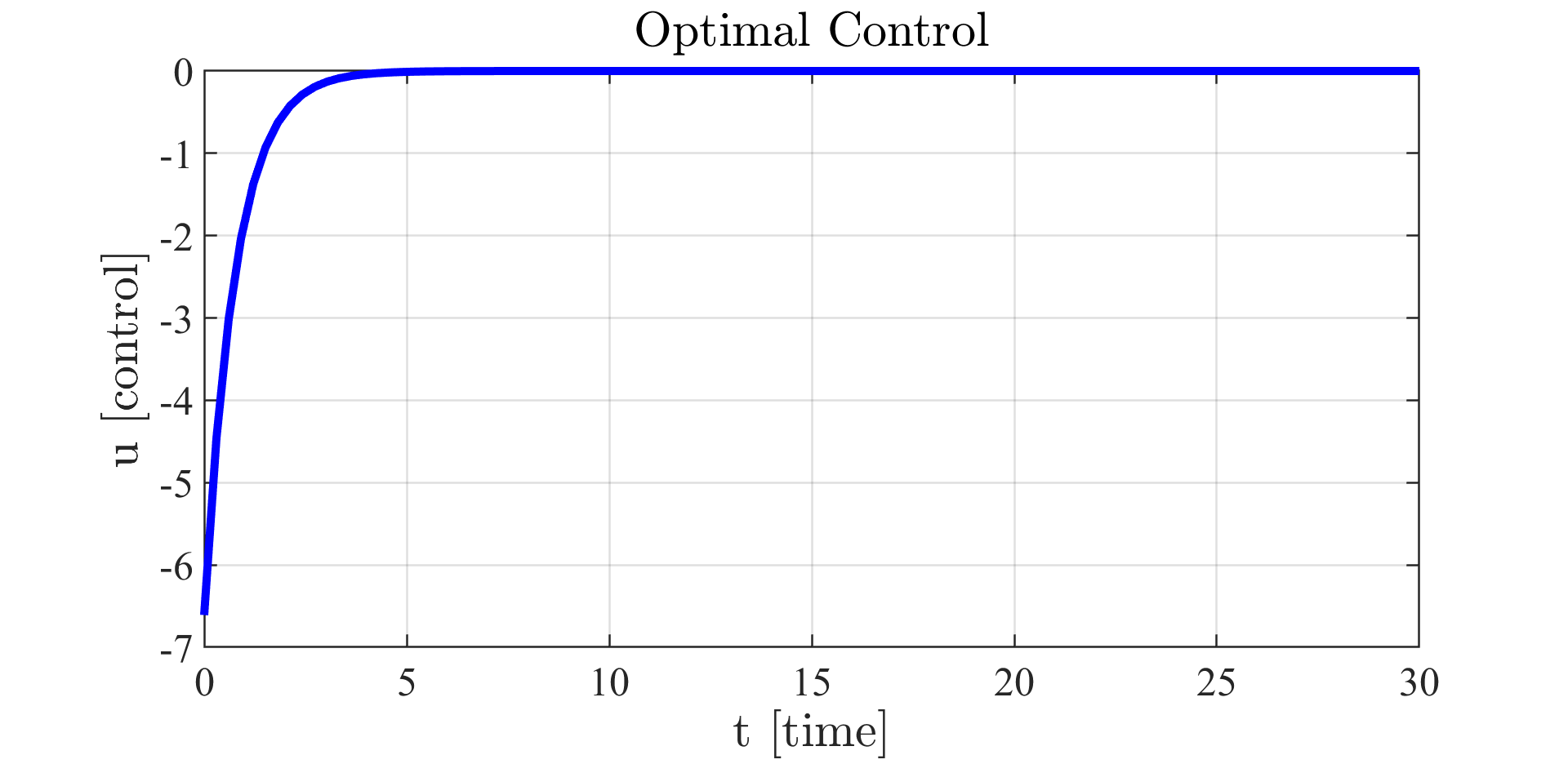}\\
		\caption{Optimal control for \eqref{functional_ocp_ex_heat_N} subject to \eqref{linear_ocp_heat_N}, with $\Omega=(0,10)$, $N=16$, $c\equiv -(2\pi/10)^2-1$, $x_{\mbox{\tiny{con}}} = L/3$, $x_{\mbox{\tiny{obs}}} = L/2$, target $z\equiv 1$ and initial datum $x_0=[1,\dots,1]$.}\label{controlheat}
	\end{center}
\end{figure}
	\begin{figure}[htp]
	\begin{center}
		% replace aims_logo.eps by your figure file name
		\includegraphics[width=12cm]{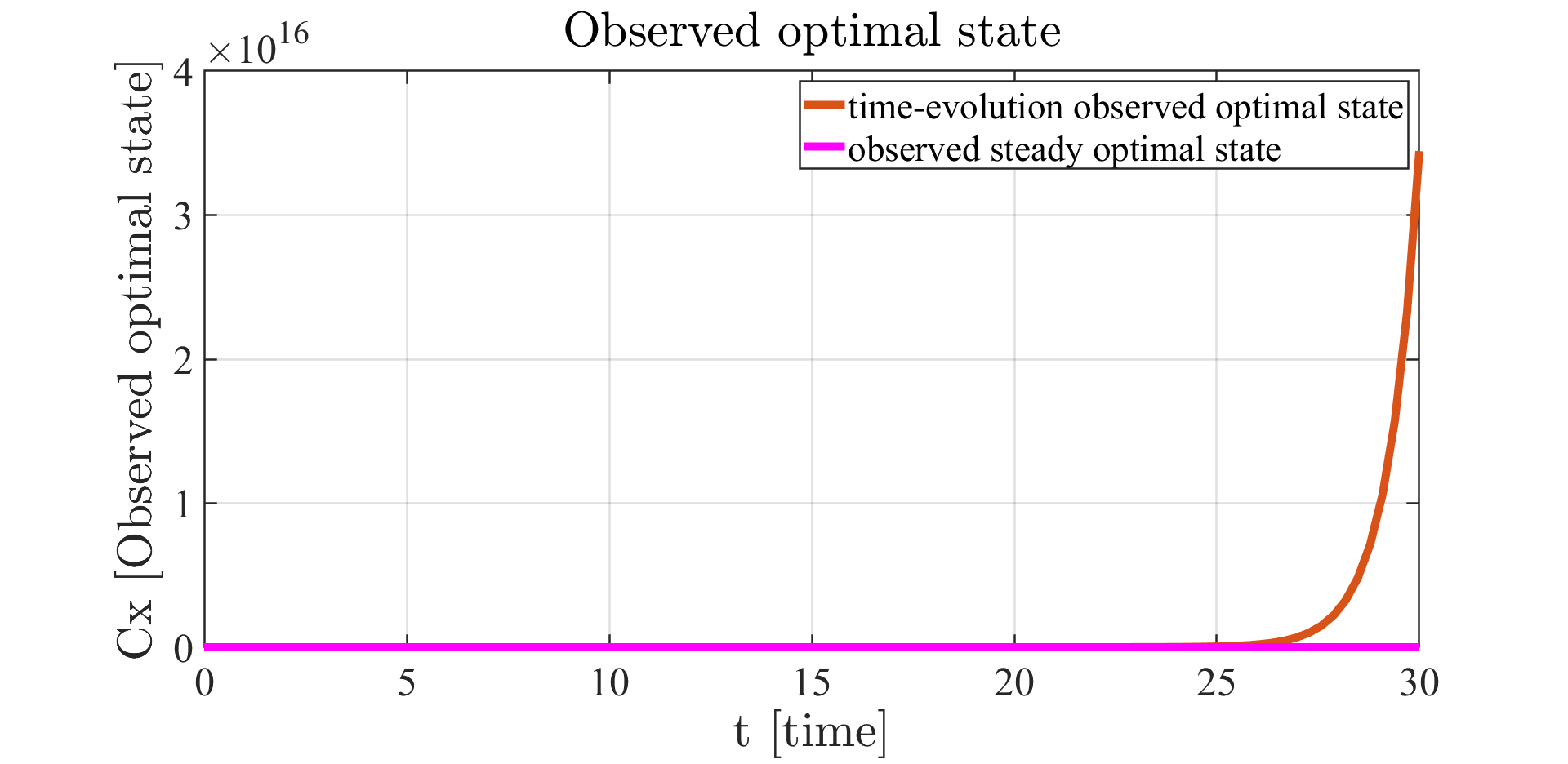}\\
		\caption{Observed optimal state for \eqref{functional_ocp_ex_heat_N} subject to \eqref{linear_ocp_heat_N}, with $\Omega=(0,10)$, $N=16$, $c\equiv -(2\pi/10)^2-1$, $x_{\mbox{\tiny{con}}} = L/3$, $x_{\mbox{\tiny{obs}}} = L/2$, target $z\equiv 1$ and initial datum $x_0=[1,\dots,1]$.}\label{observedstateheat}
	\end{center}
\end{figure}

\newpage

\subsection{Pointwise control of the wave equation}
\label{subsec:1.3}

Let $\Omega$ be a connected bounded open set of $\mathbb{R}^n$, $n =1,2,3$, with $C^{\infty}$ boundary.

As in \cite[section 1.3]{doi:10.1137/1.9781611970982.ch1}, we consider the wave equation controlled from $x_{\mbox{\tiny{con}}}\in\Omega$
\begin{equation}\label{wave_internal_1}
\begin{cases}
y_{tt}-\Delta y=v(t)\delta(x-x_{\mbox{\tiny{con}}})\hspace{0.6 cm} & \mbox{in} \hspace{0.10 cm}(0,T)\times \Omega\\
y=0  & \mbox{on}\hspace{0.10 cm} (0,T)\times \partial \Omega\\
y(0,x)=y^0_0(x), \ y_t(0,x)=y^1_0(x)  & \mbox{in} \hspace{0.10 cm}\Omega\\
\end{cases}
\end{equation}
where $y=y(t,x)$ is the state, while $v=v(t)$ is the control whose action is localized on the point $x_{\mbox{\tiny{con}}}$ by means of  multiplication with the Dirac delta $\delta(x-x_{\mbox{\tiny{con}}})$. The well posedeness of the above equation was analyzed in \cite[subsection 1.3.2]{doi:10.1137/1.9781611970982.ch1}, by using transposition techniques. For any $y_0\in H^1_0(\Omega)$, $y_1\in L^2(\Omega)$ and $v\in L^2(0,T)$, there exists a unique solution by transposition to \eqref{wave_internal_1}, with initial datum $(y_0,y_1)$ and control $v$.

We suppose that the spectrum $\left\{\lambda_k\right\}$ of the Dirichlet laplacian is simple. Let $\left\{\phi_k\right\}$ be an orthonormal basis of $L^2(\Omega)$, such that $-\Delta \phi_k=\lambda_k\phi_k$.

Fix $N\in\mathbb{N}\setminus \left\{0\right\}$. Set
\begin{equation}\label{diagonal_N_2}
D_N\coloneqq \begin{tikzpicture}[baseline=(current bounding box.center)]
\matrix (A) [matrix of math nodes,nodes in empty cells,right delimiter={]},left delimiter={[} ]{
	\lambda_1  &   &             &    &    &    &    &             &    &    \\
	&    &             &    &    &    &    & \textbf{0}  &    &    \\
	&    &             &    &    &    &    &             &    &    \\
	&    &             &    &    &    &    &             &    &    \\
	&    &             &    &    &    &    &             &    &    \\
	&    &             &    &    &    &    &             &    &    \\
	&    &             &    &    &    &    &             &    &    \\
	&    &             &    &    &    &    &             &    &    \\
	&    & \textbf{0}  &    &    &    &    &             &    &  \\
	&    &             &    &    &    &    &             &    & \lambda_N  \\
} ;
\draw[loosely dotted] (A-1-1)--(A-10-10);
\end{tikzpicture}.
\end{equation}
The projection of \eqref{wave_internal_1} onto the finite dimensional space $\mbox{span} \left\{\phi_1,\dots,\phi_N\right\}$ reads as
\begin{equation}\label{linear_ocp_wave_N}
\begin{dcases}
\dot{x}=A_Nx+B_Nu\hspace{2.8 cm} & \mbox{in} \hspace{0.10 cm}(0,T)\\
x(0)=x_0,
\end{dcases}
\end{equation}
where $A_N$ is an $2N\times 2N$ matrix
\begin{equation}\label{wave_A_N}
A=\begin{pmatrix}
0&I_N\\
-D_N&0\\
\end{pmatrix},
\end{equation}
the control operator $B_N$ is an $2N\times 1$ matrix
\begin{equation}\label{control_oper_N_2}
B_N=
\begin{tikzpicture}[baseline=(current bounding box.center)]
\matrix (b) [matrix of math nodes,nodes in empty cells,right delimiter={]},left delimiter={[} ]{
	0  \\
	\\
	\\
	\\
	\\
	\\
	\\
	\\
	\\
	0  \\
	\phi_1(x_{\mbox{\tiny{con}}})  \\
	\phi_2(x_{\mbox{\tiny{con}}})  \\
	\\
	\\
	\\
	\\
	\\
	\\
	\\
	\phi_N(x_{\mbox{\tiny{con}}})  \\
} ;
\draw[loosely dotted] (b-1-1)-- (b-10-1);
\draw[loosely dotted] (b-13-1)-- (b-20-1);
\end{tikzpicture}
\end{equation}
and the initial datum
\begin{equation*}
x_0=\left[\int_{\Omega}y_0\phi_1dx,\dots,\int_{\Omega}y_0\phi_Ndx;\int_{\Omega}y_1\phi_1dx,\dots,\int_{\Omega}y_1\phi_Ndx\right].
\end{equation*}

We consider the optimal control problem:
\begin{equation}\label{functional_ocp_ex_wave_N}
\min_{u\in L^2(0,T)}J^{T}(u)=\frac12 \int_{0}^T \left[|u(t)|^2+\|C_Nx(t)-z\|^2\right] dt,
\end{equation}
where $x$ is the state solution to \eqref{linear_ocp_wave_N}, with control $u$ and initial datum $x_0$, the scalar $z$ is a running target and
\begin{equation*}
C_N\coloneqq \left[\phi_1(x_{\mbox{\tiny{obs}}}),\dots, \phi_N(x_{\mbox{\tiny{obs}}});0,\dots,0\right],
\end{equation*}
namely the state is observed on $x_{\mbox{\tiny{obs}}}\in\Omega$.

We have the following result.
\begin{proposition}\label{prop_point_wave}
	The triplet $(A_N,B_N,C_N)$ enjoys $C_N$-turnpike if and only if for any $i\in \left\{1,\dots, N\right\}$ such that $\phi_i(x_{\mbox{\tiny{obs}}})\neq 0$, we have $\phi_i(x_{\mbox{\tiny{con}}})\neq 0$.
\end{proposition}
\begin{proof}[Proof of Proposition \ref{prop_point_wave}]
	As in the proof of Proposition \ref{prop_point_heat}, we define the observable subspace
	\begin{equation}\label{def_W_N_wave}
	W_N \coloneqq NO(C_N,A_N)^{\perp}=\sum_{i=0}^{n-1}\mbox{Range}\left(\left(A_N^*\right)^iC_N^*\right)
	\end{equation}
	and the stabilizable subspace
	\begin{equation*}
	S(A_N,B_N)= \sum_{i=0}^{n-1}\mbox{Range}\left(A_N^iB_N\right)+\mathscr{L}^{-}(A_N),
	\end{equation*}
	where in this case $A_N$ and $B_N$ are given respectively by \eqref{wave_A_N} and \eqref{control_oper_N_2}.
	The $C_N$-stabilizability of $(A_N,B_N)$ is equivalent to the inclusion
	\begin{equation*}
	W_N\subseteq S(A_N,B_N).
	\end{equation*}
	On the one hand, for any natural $i\geq 0$, we have
	\begin{equation*}
	\left(A_N^*\right)^{2i}C_N^*=
	\begin{tikzpicture}[baseline=(current bounding box.center)]
\matrix (b) [matrix of math nodes,nodes in empty cells,right delimiter={]},left delimiter={[} ]{
	(-\lambda_1)^i\phi_1(x_{\mbox{\tiny{con}}})  \\
	\\
	\\
	\\
	\\
	\\
	\\
	\\
	\\
	(-\lambda_N)^i\phi_N(x_{\mbox{\tiny{con}}})  \\
	0  \\
	\\
	\\
	\\
	\\
	\\
	\\
	\\
	\\
	0  \\
} ;
\draw[loosely dotted] (b-1-1)-- (b-10-1);
\draw[loosely dotted] (b-11-1)-- (b-20-1);
\end{tikzpicture}.
\end{equation*}
	and
	\begin{equation*}
	\left(A_N^*\right)^{2i+1}C_N^*=
\begin{tikzpicture}[baseline=(current bounding box.center)]
\matrix (b) [matrix of math nodes,nodes in empty cells,right delimiter={]},left delimiter={[} ]{
	0  \\
	\\
	\\
	\\
	\\
	\\
	\\
	\\
	\\
	0  \\
	(-\lambda_1)^i\phi_1(x_{\mbox{\tiny{con}}})  \\
	\\
	\\
	\\
	\\
	\\
	\\
	\\
	\\
	(-\lambda_N)^i\phi_N(x_{\mbox{\tiny{con}}})  \\
} ;
\draw[loosely dotted] (b-1-1)-- (b-10-1);
\draw[loosely dotted] (b-11-1)-- (b-20-1);
\end{tikzpicture}
	\end{equation*}
	On the other hand, for any natural $i\geq 0$, we have
	\begin{equation*}
\left(A_N\right)^{2i}B_N=
\begin{tikzpicture}[baseline=(current bounding box.center)]
\matrix (b) [matrix of math nodes,nodes in empty cells,right delimiter={]},left delimiter={[} ]{
		0  \\
	\\
	\\
	\\
	\\
	\\
	\\
	\\
	\\
	0  \\
	(-\lambda_1)^i\phi_1(x_{\mbox{\tiny{con}}})  \\
	\\
	\\
	\\
	\\
	\\
	\\
	\\
	\\
	(-\lambda_N)^i\phi_N(x_{\mbox{\tiny{con}}})  \\
} ;
\draw[loosely dotted] (b-1-1)-- (b-10-1);
\draw[loosely dotted] (b-11-1)-- (b-20-1);
\end{tikzpicture}
\end{equation*}
	and
	\begin{equation*}
	\left(A_N\right)^{2i+1}B_N=
	\begin{tikzpicture}[baseline=(current bounding box.center)]
	\matrix (b) [matrix of math nodes,nodes in empty cells,right delimiter={]},left delimiter={[} ]{
		(-\lambda_1)^i\phi_1(x_{\mbox{\tiny{con}}})  \\
		\\
		\\
		\\
		\\
		\\
		\\
		\\
		\\
		(-\lambda_N)^i\phi_N(x_{\mbox{\tiny{con}}})  \\
		0  \\
		\\
		\\
		\\
		\\
		\\
		\\
		\\
		\\
		0  \\
	} ;
	\draw[loosely dotted] (b-1-1)-- (b-10-1);
	\draw[loosely dotted] (b-11-1)-- (b-20-1);
	\end{tikzpicture}.
	\end{equation*}
	Hence, since the spectrum of the Dirichlet laplacian is simple, the $(C_N,A_N)$ observable subspace is
	\begin{equation*}
	W_N=\left\{w\in\mathbb{R}^{2N} \ | \ w_i=0 \ \mbox{and} \ w_{i+1}=0 \ \mbox{if} \ \phi_i(x_{\mbox{\tiny{obs}}})=0 \right\}
	\end{equation*}
	and the $(A_N,B_N)$ stabilizable subspace reads as
	\begin{equation*}
	S(A_N,B_N)=\left\{w\in\mathbb{R}^{2N} \ | \ w_i=0 \ \mbox{and} \ w_{i+1}=0 \ \mbox{if} \ \phi_i(x_{\mbox{\tiny{con}}})=0 \right\}.
	\end{equation*}
	Then, the inclusion $W_N\subseteq S(A_N,B_N)$ holds if and only if
	\begin{equation*}
	\left\{i\in \left\{1,\dots,N\right\} \ | \ \phi_i(x_{\mbox{\tiny{obs}}})\neq 0\right\}\subset \left\{i\in \left\{1,\dots,N\right\} \ | \ \phi_i(x_{\mbox{\tiny{con}}})\neq 0 \right\}
	\end{equation*}
	
	Then, by Proposition \ref{prop_eqC}, the triplet $(A_N,B_N,C_N)$ enjoys $C_N$-turnpike if and only if for any $i\in \left\{1,\dots, N\right\}$ such that $\phi_i(x_{\mbox{\tiny{obs}}})\neq 0$, $\phi_i(x_{\mbox{\tiny{con}}})\neq 0$.
\end{proof}

We carried out some numerical simulations using the interior-point optimization routine \verb!IPOpt! (see \cite{IDO} and \cite{waechter2009introduction}) coupled with \verb!AMPL! \cite{FAP}, which serves as modelling language and performs the automatic differentiation. In our notation, we chose domain $\Omega=(0,10)$, $N=16$, $x_{\mbox{\tiny{con}}} = L/2$, $x_{\mbox{\tiny{obs}}} = L/2$, target $z\equiv 1$ and initial datum $x_0=[1,\dots,1]$. The results are illustrated in figures \ref{controlwave} and \ref{observedstatewave}.
\begin{figure}[htp]
	\begin{center}
		% replace aims_logo.eps by your figure file name
		\includegraphics[width=12cm]{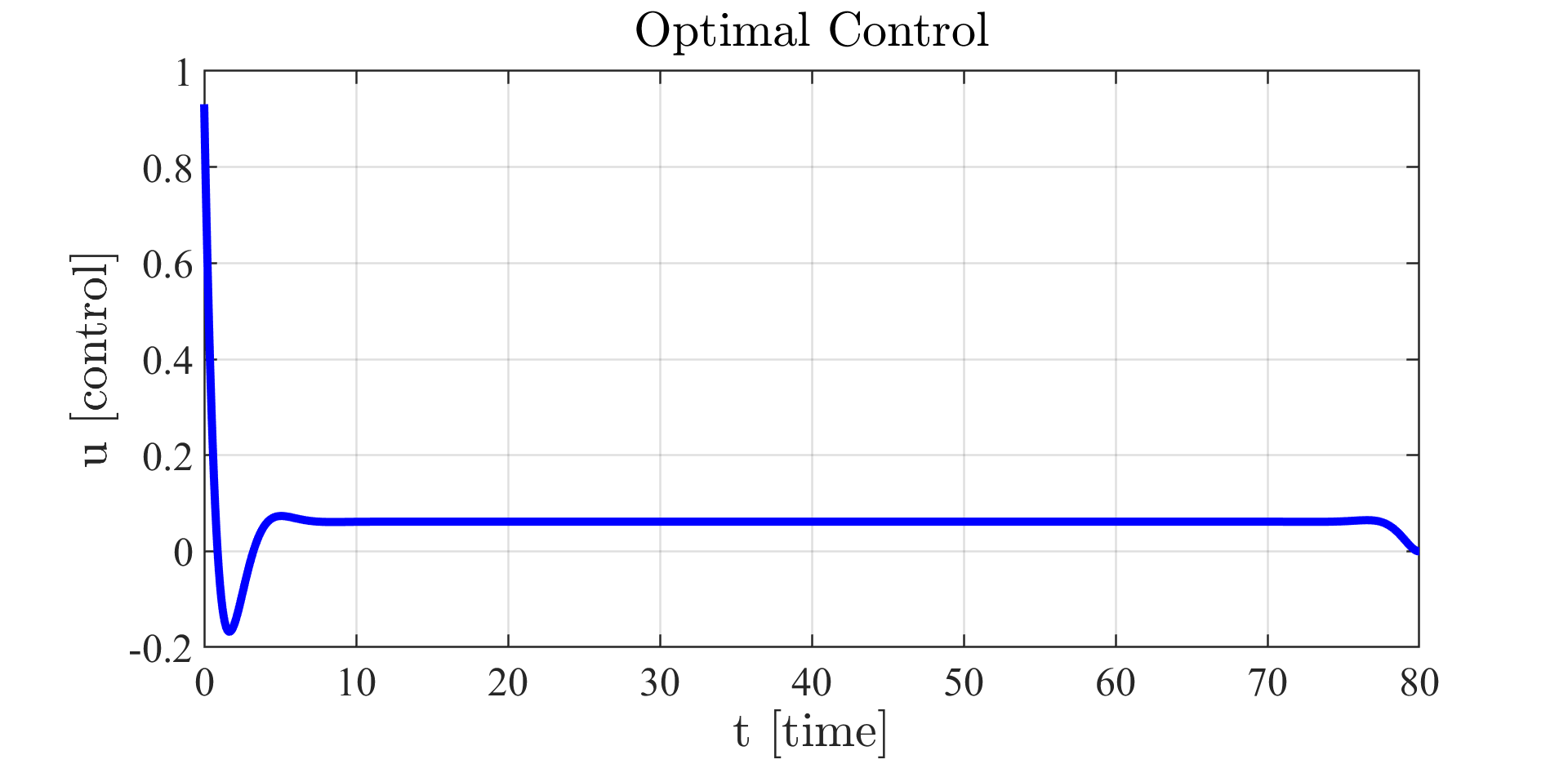}\\
		\caption{Optimal control for \eqref{linear_ocp_wave_N}-\eqref{functional_ocp_ex_wave_N}, with $\Omega=(0,10)$, $N=16$, $x_{\mbox{\tiny{con}}} = L/2$, $x_{\mbox{\tiny{obs}}} = L/2$, $z\equiv 1$ and $x_{\mbox{\tiny{con}}}=[1,\dots,1]$.}\label{controlwave}
	\end{center}
\end{figure}
\begin{figure}[htp]
	\begin{center}
		% replace aims_logo.eps by your figure file name
		\includegraphics[width=12cm]{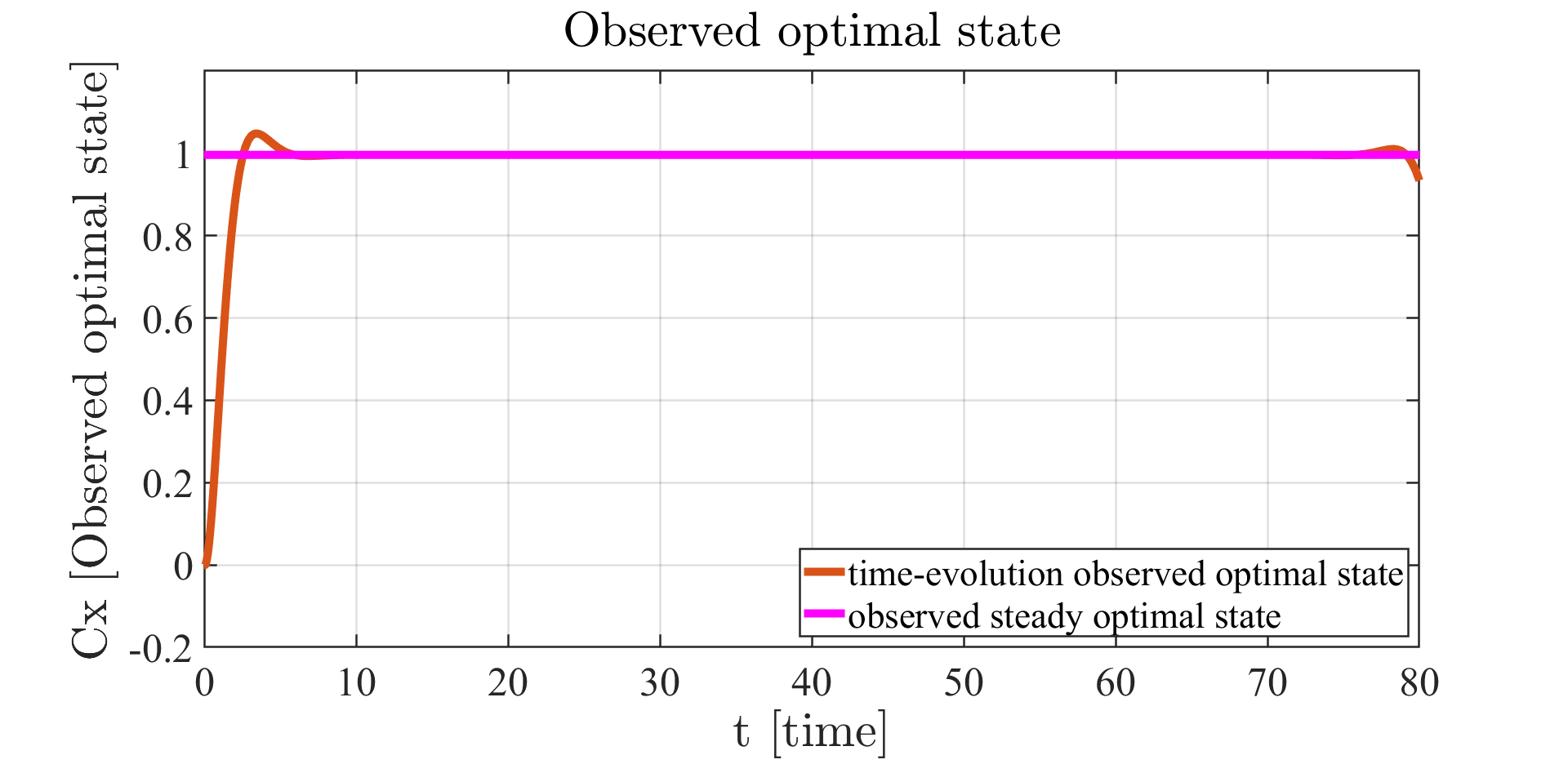}\\
		\caption{Observed optimal state for \eqref{linear_ocp_wave_N}-\eqref{functional_ocp_ex_wave_N}, with $\Omega=(0,10)$, $N=16$, $x_{\mbox{\tiny{con}}} = L/2$, $x_{\mbox{\tiny{obs}}} = L/2$, $z\equiv 1$ and $x_{\mbox{\tiny{con}}}=[1,\dots,1]$.}\label{observedstatewave}
	\end{center}
\end{figure}

\newpage

\subsection{Proof of Proposition \ref{prop_eqC}}
\label{subsec:1.4}

The key-tool for our analysis of the control problem \eqref{linear_ocp}-\eqref{functional_ocp} is the Kalman decomposition (see, e.g. \cite[section 3.3]{zhou1996robust}). In the notation of \cite{callier1995convergence}, $\mathscr{L}^{0,+}(A)$ denote the $A$-invariant subspace of $\mathbb{R}^n$ spanned by the generalized eigenvectors of $A$ corresponding to eigenvalues $\lambda$ of $A$ such that $\mbox{Re}(\lambda)\geq 0$. We decompose the state space into an detectable part and an undetectable one
\begin{equation*}
\mathbb{R}^n=W\oplus NO^{0+}(C,A),
\end{equation*}
where
\begin{equation}\label{def_NO}
NO^{0+}(C,A) = \bigcap_{i=0}^{n-1}\ker\left(CA^i\right)\bigcap \mathscr{L}^{0,+}(A)
\end{equation}
and
\begin{equation}\label{def_W}
W \coloneqq NO^{0+}(C,A)^{\perp}.
\end{equation}
The matrix associated to the orthogonal projection onto $W$ is denoted by $D$, while the matrix associated to the projection onto $NO^{0+}(C,A)$ is indicated by $R$. Hence, for any $x\in \mathbb{R}^n$,
\begin{equation*}
	x=Dx+Rx.
\end{equation*}
where $Dx\in W$ and $Rx\in NO^{0+}(C,A)$. By definition \eqref{def_NO}, we realize that
\begin{equation}\label{obs_new_var}
	Cx=CDx+CRx=CDx.	
\end{equation}
Moreover, by \eqref{def_NO}, we have $ANO^{0+}(C,A)\subseteq NO^{0+}(C,A)$. Then, $ARx\in NO^{0+}(C,A)$, whence
\begin{equation}\label{proj_AR}
	RARx=ARx \hspace{0.3 cm}\mbox{and} \hspace{0.3 cm} DARx=0.	
\end{equation}

In the notation of \cite{callier1995convergence}, $\mathscr{L}^{-}(A)$ denote the $A$-invariant subspace of $\mathbb{R}^n$ spanned by the generalized eigenvectors of $A$ corresponding to eigenvalues $\lambda$ of $A$ such that $\mbox{Re}(\lambda)<0$ and the stabilizable subspace
\begin{equation*}
	S(A,B)\coloneqq \sum_{i=0}^{n-1}\mbox{Range}\left(A^iB\right)+\mathscr{L}^{-}(A).
\end{equation*}

An essential tool for the proof of Proposition \ref{prop_eqC} is the following Lemma.
\begin{lemma}\label{lemma_Cstabeqdef}
	Let $A\in\mathcal{M}_{n\times n}(\mathbb{R})$, $B\in\mathcal{M}_{n\times m}(\mathbb{R})$ and $C\in\mathcal{M}_{n\times n}(\mathbb{R})$. $(A,B)$ is $C$-stabilizable if and only if for any initial datum $x_0\in \mathbb{R}^n$, there exists a control $u\in L^2(0,+\infty;\mathbb{R}^m)$, such that
	\begin{equation}\label{Cintfinite}
	\int_0^{\infty}\|Cx\|^2dt<+\infty,
	\end{equation}
	$x$ being the solution to \eqref{linear_ocp}, with initial datum $x_0$ and control $u$.
\end{lemma}
The above Lemma follows from \cite[Remark 2.2 page 24]{BRC} applied to the detectable space $W$, introduced in \eqref{def_W}.
%\begin{proof}[Proof of Lemma \ref{lemma_Cstabeqdef}]
%%	\textit{Step 1} \  \textbf{$(A,B)$ $C$ stabilizable implies \eqref{Cintfinite}}\\
%%	[CHECK]
%%	\textit{Step 2} \  \textbf{\eqref{Cintfinite} implies $(A,B)$ $C$ stabilizable}\\
%	
%\end{proof}

We are now ready to prove Proposition \ref{prop_eqC}.

\begin{proof}[Proof of Proposition \ref{prop_eqC}.]
	\textit{Step 1} \  \textbf{Necessity of the $C$-stabilizability of $(A,B)$}\\
	Suppose $(A,B,C)$ enjoys $C$-turnpike. Then, taking target $z=0$, for any initial datum $x_0\in\mathbb{R}^n$, we have
	\begin{equation*}
	\|u^T(t)\|+\|Cx^T(t)\|\leq K\left[\exp\left(-\mu t\right)+\exp\left(-\mu\left(T-t\right)\right)\right],
	\end{equation*}
	$(u^T,x^T)$ being the optimal pair for \eqref{linear_ocp}-\eqref{functional_ocp}, with target $z=0$ and initial datum $x_0$, whence
	\begin{equation}\label{bound_functional}
	J^{T}(u^T) \leq K,
	\end{equation}
	where $K$ is independent of the time horizon $T$. By Banach-Alaoglu Theorem, there exists $u^{\infty}\in L^2((0,+\infty);\mathbb{R}^n)$, such that, up to subsequences,
	\begin{equation*}
		u^T\underset{T\to +\infty}{\longrightarrow}u^{\infty},
	\end{equation*}
	weakly in $L^2((0,+\infty);\mathbb{R}^m)$. We denote by $x^{\infty}$ the solution to \eqref{linear_ocp}, with initial datum $x_0$ and control $u^{\infty}$. Arbitrarily fix $S>0$. By definition of weak convergence, up to subsequences
		\begin{equation*}
	x^T\underset{T\to +\infty}{\longrightarrow}x^{\infty},
	\end{equation*}
	weakly in $L^2((0,S);\mathbb{R}^m)$. By lower-semicontinuity of the norm with respect to the weak convergence and \eqref{bound_functional}, for any $S>0$, we have
	%for any $T\in (S,+\infty)$
	\begin{equation*}
	\int_{0}^{S}\|Cx^{\infty}\|^2dt\leq \liminf_{T\to +\infty}\frac12\int_{0}^{S}\left[\|u^T(t)\|^2+\|Cx^T(t)\|^2\right]dt\leq \liminf_{T\to +\infty}J^T(u^T)\leq K,
	\end{equation*}
	whence, by the arbitrariness of $S$,
	\begin{equation*}
\int_{0}^{\infty}\|Cx^{\infty}\|^2dt\leq K<+\infty.
\end{equation*}	
	Then, by Lemma \ref{lemma_Cstabeqdef}, $(A,B)$ is $C$-stabilizable.\\
	\textit{Step 2} \  \textbf{We rewrite the functional employing Kalman decomposition}\\
	Assume $(A,B)$ is $C$-stabilizable. Take any initial datum $x_0\in \mathbb{R}^n$ and control $u\in L^2(0,T;\mathbb{R}^m)$. Let $x$ be the corresponding solution to the state equation
	\begin{equation}\label{linear_ocp_0.1}
	\begin{dcases}
	\dot{x}=Ax+Bu\hspace{2.8 cm} & \mbox{in} \hspace{0.10 cm}(0,T)\\
	x(0)=x_0.
	\end{dcases}
	\end{equation}
	Set $y\coloneqq Dx$. By \eqref{proj_AR},
	\begin{equation*}
	\dot{y}=\frac{d}{dt}\left[Dx\right]=DAx+DBu=DADx+DARx+DBu=DADx+DBu=DAy+DBu,
	\end{equation*}
	whence
	\begin{equation}\label{linear_ocp_0.6}
	\begin{dcases}
	\dot{y}=DAy+DBu\hspace{2.8 cm} & \mbox{in} \hspace{0.10 cm}(0,T)\\
	y(0)=Dx_0.
	\end{dcases}
	\end{equation}
	Furthermore, by \eqref{proj_AR},
	\begin{equation*}
	Cx=CDx+CRx=C{y}.
	\end{equation*}

	Then, the functional $J^T$ introduced in \eqref{linear_ocp}-\eqref{functional_ocp} can be rewritten as
	\begin{equation}\label{functional_ocp_transformed}
	J^{T}(u)=\frac12 \int_{0}^T \left[\|u(t)\|^2+\|C{y}(t)-z\|^2\right] dt,
	\end{equation}
	where:
	\begin{equation}\label{linear_state_equation_transformed}
	\begin{dcases}
	\dot{y}=DAy+DBu\hspace{2.8 cm} & \mbox{in} \hspace{0.10 cm}(0,T)\\
	y(0)=Dx_0.
	\end{dcases}
	\end{equation}
	\textit{Step 3} \ \textbf{Sufficiency of $C$-stabilizability of $(A,B)$}\\
	Let $u^T$ be the optimal control and $y^T$ be the optimal state for the optimal control problem \eqref{linear_state_equation_transformed}-\eqref{functional_ocp_transformed}.
	
	By construction, $\left(DA,C\right)$ is detectable on $W$ and, since $(A,B)$ is $C$-stabilizable, 
	%[CHECK]
	$\left(DA,DB\right)$ is stabilizable on $W$. Then, by \cite[Corollary 3.2]{TGS} applied to \eqref{linear_state_equation_transformed}-\eqref{functional_ocp_transformed}, we have
	\begin{equation*}
	\|u^T(t)-\overline{u}\|+\left\|Dx^T(t)-D{\overline{x}}\right\|\leq K\left[\exp\left(-\mu t \right)+\exp\left(-\mu\left(T-t\right)\right)\right],
	\end{equation*}
	$K$ and $\mu >0$ being independent of the time horizon. Then, the triplet $(A,B,C)$ enjoys $C$-turnpike, as desired.
\end{proof}

\section{Fixed endpoint problem}
\label{sec:2}

In this section, we consider an optimal control problem for \eqref{linear_ocp} with arbitrarily prescribed terminal state. Let an initial datum $x_0\in\mathbb{R}^n$ and a final target $x_1\in\mathbb{R}^n$ be given.

In the notation of section \ref{sec:1}, assume $(A,B)$ is controllable. Consider the control system
\begin{equation}\label{linear_ocp_2}
\begin{dcases}
\dot{x}=Ax+Bu\hspace{2.8 cm} & \mbox{in} \hspace{0.10 cm}(0,T)\\
x(0)=x_0, \ x(T)=x_1.
\end{dcases}
\end{equation}

We introduce the set of admissible controls
\begin{equation}\label{ocp_both_endpoints_fixed}
\mathscr{U}_{\mbox{\tiny{ad}}}\coloneqq \left\{u\in L^2(0,T;\mathbb{R}^m) \ | \ \mbox{there exists a solution}\hspace{0.3 cm} x \hspace{0.3 cm} \mbox{to \eqref{linear_ocp_2}}\right\}.
\end{equation}

We consider the linear quadratic optimal control problem:
\begin{equation}\label{functional_ocp_2}
\min_{u\in \mathscr{U}_{\mbox{\tiny{ad}}}}J^{T}(u)=\frac12 \int_{0}^T \left[\|u(t)\|^2+\|Cx(t)-z\|^2\right] dt,
\end{equation}
where $z\in \mathbb{R}^n$ is a running target. By the Direct Methods in the Calculus of Variations and strict convexity, the above problem admits a unique optimal control denoted by $u^T$. The optimal state is denoted by $x^T$. Furthermore, by strict convexity, the optimal control $u^T=-B^*p^T$, where $(x^T,p^T)$ is the unique solution to the optimality system

\begin{equation}\label{OS_controllability}
\begin{dcases}
\dot{x}^T(t)=Ax^T(t)-BB^*p^T(t)\hspace{1 cm}& t\in (0,T)\\
-\dot{p}^T(t)=A^*p^T(t)+C^*\left(Cx^T(t)-z\right)&t\in (0,T)\\
x^T(0)=x_0\\
x^T(T)=x_1.
\end{dcases}
\end{equation}

As in the free endpoint problem, the steady problem is a minimization problem in finite dimension under linear constraints
The corresponding steady problem reads as
\begin{equation}\label{steady_functional_linear_quadratic_fixed_endpoint}
\min_{(x,u)}J_s(x,u)=\frac12 \|u\|^2+\frac{1}{2}\|Cx-z\|,\hspace{0.6 cm}\mbox{with the constraint} \hspace{0.6 cm}0=Ax+Bu.
\end{equation}
This problem has been analyzed in section \ref{sec:1}.

As we anticipated, the fixed endpoint case is more delicate. Indeed,
\begin{itemize}
	\item in the functional, we penalize only the observable part of the state;
	\item in the definition of the set of admissible control, we impose a final condition $x^T(T)=x_1$, which involves the full state, including the unobservable part.
\end{itemize}
We cannot expect the classical turnpike property to be valid, without any additional assumptions on $(A,C)$. Indeed, take $A$ skew adjoint and $C=0$. The resulting optimal control is oscillatory, thus violating the turnpike property.

The turnpike property is verified by the full state and adjoint state if $(A,C)$ fulfills the \textit{weak Hautus test}, introduced in definition \eqref{weak_Hautus test} below. In subsection \ref{subsec:2.2}, inspired by \cite{faulwasser2019towards}, we study velocity turnpike in case the Hamiltonian matrix has imaginary eigenvalues.

\subsection{The sufficiency of the weak Hautus test for exponential turnpike}
\label{subsec:2.1}

\begin{definition}\label{weak_Hautus test}
	Take $(A,C)$ as in \eqref{linear_ocp_2}-\eqref{functional_ocp_2}. The pair $(A,C)$ is said to satisfy the \textit{weak Hautus test} if
	\begin{equation}\label{weakHtest}
		\mbox{rank}\begin{bmatrix}
		A-i\beta I\\
		C
		\end{bmatrix}=n,\hspace{0.68 cm}\forall \ i\beta\in \mbox{sp}\left(A\right),
	\end{equation}
	where $\mbox{sp}\left(A\right)$ denotes the spectrum of $A$, $i$ stands for the imaginary unit and $\beta\in\mathbb{R}$.
\end{definition}

Note that the difference with respect to the classical Hautus test (see, e.g. \cite[Proposition 1.5.1]{OCO}) is that the above rank condition has to be checked only for purely imaginary eigenvalues $i\beta$. Namely, only eigenvectors of imaginary eigenvalues of $A$ are required to be observable.

\begin{proposition}\label{prop_controllability_case}
	Suppose $(A,B)$ is controllable and the weak Hautus test \eqref{weakHtest} is satisfied. Take $T>2$. Let $u^T$ be an optimal control for \eqref{linear_ocp_2}-\eqref{functional_ocp_2} and $x^T$ be the optimal state. There exist $T$-independent $K$ and $\mu >0$ such that
	\begin{equation}\label{ocp_both_endpoints_fixed_6}
		\|u^T(t)-\overline{u}\|+\|x^T(t)-\overline{x}\|\leq K\left[\exp\left(-\mu t \right)+\exp\left(-\mu\left(T-t\right)\right)\right],
	\end{equation}
	where $\left(\overline{u},\overline{x}\right)$ is the unique solution to the steady problem \eqref{steady_functional_linear_quadratic_fixed_endpoint}.
\end{proposition}

The proof of the sufficiency of the weak Hautus test \eqref{weakHtest} (available at the end of the present section) is based on Lemma \ref{lemma_Ham_spectrum}, Lemma \ref{lemma_infnorm_bound} and Lemma \ref{lemma_1}, concerning the properties of the optimality system \eqref{OS_controllability} and its associated matrix, the so-called Hamiltonian matrix
%[CHECK]
\begin{equation}\label{Ham}
\mbox{Ham}\coloneqq \begin{bmatrix}
A&-BB^*\\
-C^*C&-A^*
\end{bmatrix}.
\end{equation}
Lemma \ref{lemma_Ham_spectrum} is well-known in the literature (see e.g. \cite[Lemma 8]{kucera1972contribution}). However, we provide the proof for the reader's convenience.

\begin{lemma}\label{lemma_Ham_spectrum}
	Consider the Hamiltonian matrix $\mbox{Ham}$ introduced in \eqref{Ham}. Assume $(A,B)$ is stabilizable. We have $\mathscr{L}^{0}(H)=\left\{0\right\}$ if and only if the weak Hautus test is satisfied.
\end{lemma}
\begin{proof}[Proof of Lemma \ref{lemma_Ham_spectrum}]
	%\textit{Step 1} \  \textbf{Conclusion}\\
	Since $(A,B)$ is stabilizable, the Algebraic Riccati Equation
	\begin{equation}\label{ARE}
	\widehat{E}A+A^*\widehat{E}-\widehat{E}BB^*\widehat{E}+C^*C=0 \hspace{1 cm} \mbox{(ARE)}
	\end{equation}
	admits a unique antistrong solution $\widehat{E}$, a symmetric and positive semidefinite matrix, such that $A_+\coloneqq A-BB^*\widehat{E}$ has all eigenvalues with nonpositive real parts (see e.g. \cite{kuvcera1991algebraic} and references therein). As in \cite{roth1950matric} and \cite[formula (3.8) page 57]{kuvcera1991algebraic}, set
	\begin{equation*}
	\Lambda \coloneqq 		\begin{bmatrix}
	I_n&0\\
	-\widehat{E}&I_n.
	\end{bmatrix}
	\end{equation*}
	Using \eqref{ARE}, we have\footnote{
		\begin{equation*}
		\Lambda^{-1}=\begin{bmatrix}
		I_n&0\\
		\widehat{E}&I_n.
		\end{bmatrix}
		\end{equation*}
	}
	\begin{equation*}
	\Lambda \hspace{0.03 cm} \mbox{Ham} \hspace{0.03 cm} \Lambda^{-1} = \begin{bmatrix}
	A-BB^*\widehat{E}&-BB^*\\
	0&-\left(A-BB^*\widehat{E}\right)^*.
	\end{bmatrix}
	\end{equation*}
	This, together with \cite[Fact 1-(f)]{callier1995convergence}, yields the equivalence between $\mathscr{L}^{0}(H)=\left\{0\right\}$ and $NO^0(C,A)=\left\{0\right\}$, with
	\begin{equation*}
	NO^0(C,A)\coloneqq\bigcap_{i=0}^{n-1}\ker\left(CA^i\right)\cap \mathscr{L}^{0}(A).
	\end{equation*}
	To conclude, we have to prove that $NO^0(C,A)=\left\{0\right\}$ if and only if $(A,C)$ satisfies the weak Hautus test \eqref{weakHtest}. On the one hand, if $NO^0(C,A)=\left\{0\right\}$, then for any eigenvector $v$ corresponding to imaginary eigenvalue $i\beta$ of $A$ and for any $k\in \mathbb{N}\cup \left\{0\right\}$
	\begin{equation}\label{CAkeigen}
	CA^kv=\left(i\beta\right)^kCv.
	\end{equation}
	Since $NO^0(C,A)=\left\{0\right\}$, we have $Cv\neq0$, whence
	\begin{equation*}
	\mbox{rank}\begin{bmatrix}
	A-i\beta I\\
	C
	\end{bmatrix}=n,
	\end{equation*}
	as required. On the other hand, suppose \eqref{weakHtest} is verified. Suppose, by contradiction, that $NO^0(C,A)\supsetneq\left\{0\right\}$. Since $NO^0(C,A)$ is $A$-invariant,
	%see \cite[Proposition 1.5.1]{OCO}
	there exists a nonzero eigenvector $v\in NO^0(C,A)$ corresponding to an imaginary eigenvalue $i\beta$. By \eqref{CAkeigen}, this leads to $Cv=0$, which yields
	\begin{equation*}
	\mbox{rank}\begin{bmatrix}
	A-i\beta I\\
	C
	\end{bmatrix}<n,
	\end{equation*}
	so obtaining a contradiction, as desired.
\end{proof}

We now provide a global bound of the norm of the adjoint state, uniform in the time horizon $T>2$.

\begin{lemma}\label{lemma_infnorm_bound}
	Consider the control problem \eqref{linear_ocp_2}-\eqref{functional_ocp_2}. There exists $K=K(A,B,C)$ such that, for any time horizon $T>2$ and time instant $t\in [0,T]$, we have
	\begin{equation}
	\left\|p^T(t)\right\|\leq K\left[\left\|x_0\right\|+\left\|x_1\right\|+\left\|z\right\|\right]
	\end{equation}
\end{lemma}
\begin{proof}[Proof of Lemma \ref{lemma_infnorm_bound}]
	\textit{Step 1} \ \textbf{Reduction to the case running target $z=0$}\\
	By Lemma \ref{lemma_Hamkerrange}, the vector $[0;C^*z]\in \mbox{Range}\left(\mbox{Ham}\right)$, whence there exists $\left(\overline{x},\overline{p}\right)\in \mathbb{R}^n\times \mathbb{R}^n$ solving the steady optimality system
	\begin{equation}\label{steady_OS_6}
	\begin{dcases}
	A\overline{x}-BB^*\overline{p}&=0\\
	-A^*\overline{p}-C^*(C\overline{x}-z)&=0.\\
	\end{dcases}
	\end{equation}
	Introduce the perturbation variables $\tilde{x}^T \coloneqq x^T-\overline{x}$ and $\tilde{p}^T \coloneqq p^T-\overline{p}$. We realize that the pair $\left(\tilde{x}^T,\tilde{p}^T\right)$ solves
	\begin{equation}\label{OS_controllability_6}
	\begin{dcases}
	\dot{\tilde{x}}^T(t)=A\tilde{x}^T(t)-BB^*\tilde{p}^T(t)\hspace{1 cm}& t\in (0,T)\\
	-\dot{\tilde{p}}^T(t)=A^*\tilde{p}^T(t)+C^*C\tilde{x}^T(t)&t\in (0,T)\\
	\tilde{x}^T(0)=x_0-\overline{x}\\
	\tilde{x}^T(T)=x_1-\overline{x},
	\end{dcases}
	\end{equation}
	the optimality system for the control problem \eqref{linear_ocp_2}-\eqref{functional_ocp_2}, with terminal conditions $\tilde{x}^T(0)=x_0-\overline{x}$ and $\tilde{x}^T(T)=x_1-\overline{x}$. This allows us to reduce to the case $z=0$. To avoid weighting the notation, we will drop the tilde in $\tilde{x}^T$ and $\tilde{p}^T$.\\
	\textit{Step 2} \  \textbf{Upper bound for the minimum value of the functional}\\
	Consider the control
	\begin{equation}\label{picewisecontrol}
	\hat{u}(t)\coloneqq
	\begin{dcases}
	u_0(t) \quad & t\in (0,1)\\
	0 \quad & t\in (1,T-1)\\
	u_1(t-T+1) \quad & t\in (T-1,T),\\
	\end{dcases}
	\end{equation}
	where
	$u_0:(0,1)\longrightarrow \mathbb{R}^m$ drives the control system \eqref{linear_ocp_2} from $x_0$ to $0$ in time $1$ and $u_1:(0,1)\longrightarrow \mathbb{R}^m$ steers \eqref{linear_ocp_2} from $0$ to $x_1$ in time $1$. Consequently the control $\hat{u}$, steers the system from $x_0$ to $x_1$ in time $T$.
	
	Now, for $i=0,1$, the control $u_i$ operates in an interval of length one. Then, there exists some $K$, independent of $T$, such that
	\begin{equation*}
	\|u_i\|_{L^2(0,1)}\leq K\|x_i\|,
	\end{equation*}
	whence
	\begin{equation}\label{bound_tildeu}
	\|\hat{u}\|_{L^2(0,T)}\leq K\left[\|x_0\|+\|x_1\|\right].
	\end{equation}
	Now, let $\tilde{x}$ be the solution to \eqref{linear_ocp_2}, with control $\hat{u}$. By definition of $\hat{u}$, $\tilde{x}(t)=0$, for $t\in [1,T-1]$, whence
	\begin{eqnarray}\label{bound_tildex}
	\int_0^T\|\tilde{x}\|^2dt&=&\int_0^1\|\tilde{x}\|^2 dt+\int_{T-1}^T\|\tilde{x}\|^2dt\nonumber\\
	&\leq&K\left[\|x_0\|^2+\|x_1\|^2+\int_0^1\|u_0\|^2dt+\int_{T-1}^{T}\|u_1\|^2dt\right]\nonumber\\
	&\leq&K\left[\|x_0\|^2+\|x_1\|^2\right],
	\end{eqnarray}
	the constant $K$ being independent of the time horizon. Therefore, by \eqref{bound_tildeu} and \eqref{bound_tildex} and since $u^T$ is a minimizer of $J^T$,
	\begin{equation}\label{est_min_val_functional}
	J^T(u^T)\leq J^T(\hat{u})\leq K\left[\|x_0\|^2+\|x_1\|^2\right].
	\end{equation}
	\textit{Step 3} \  \textbf{Boundedness of $\|p^T(0)\|$ and $\|p^T(T)\|$}\\
	By assumptions, $(A,B)$ is controllable. Then, $(A^*,B^*)$ is observable. Therefore, by adapting the techniques of \cite[remark 2.1 page 4245]{porretta2013long}, for every $t\in [0,T]$, we have
	\begin{eqnarray}
	\|p^T(t)\|^2&\leq&K\left[\int_0^T\|B^*p^T(s)\|^2ds+\int_0^T\|C^*Cx^T(s)\|^2ds\right]\nonumber\\
	&\leq&K\left[\int_0^T\|u^T(s)\|^2ds+\int_0^T\|Cx^T(s)\|^2ds\right]\nonumber\\
	&=&KJ^T(u^T)\leq K\left[\|x_0\|^2+\|x_1\|^2\right],
	\end{eqnarray}
	with $K$ independent of $T$, as desired.\\
\end{proof}

Let $H\in \mathcal{M}_{N\times N}(\mathbb{R})$ be a square matrix. Following the notation of \cite{callier1995convergence}, $\mathscr{L}^{-}(H)$, $\mathscr{L}^{0}(H)$ and $\mathscr{L}^{+}(H)$ denote resp. the $H$-invariant subspaces of $\mathbb{R}^n$ spanned by the generalized eigenvectors of $H$ corresponding to eigenvalues $\lambda$ of $H$ such that $\mbox{Re}(\lambda)<0$, $\mbox{Re}(\lambda)=0$ and $\mbox{Re}(\lambda)> 0$. In the proof Proposition \ref{prop_controllability_case} we use the following Lemma, proved in the appendix.

\begin{lemma}\label{lemma_1}
	Let $H\in \mathcal{M}_{N\times N}(\mathbb{R})$ be a square matrix. Let $y$ be a solution to
	\begin{equation}\label{linear_nohomo}
	\dot{y}=Hy.
	\end{equation}
	Then, for every $t\in [0,T]$
	\begin{equation}\label{est_L_inf}
	\mbox{dist}\left(y(t),\mathscr{L}^{0}(H)\right)\leq K\left[\exp\left(-\mu t\right)\|y(0)\|+\exp\left(-\mu\left(T-t\right)\right)\|y(T)\|\right]
	\end{equation}
	the constants $K$ and $\mu>0$ being independent of the time horizon.
\end{lemma}

We are now in position to prove Proposition \ref{prop_controllability_case}.
\begin{proof}[Proof of Proposition \ref{prop_controllability_case}]
%	\textit{Step 1} \  \textbf{Necessity of the weak Hautus test}\\
%	Suppose the weak Hautus test is not satisfied. Then, by Lemma \ref{lemma_Ham_spectrum}, the critical subspace $\mathscr{L}^{0}(H)$ contains nonzero vectors.
%	%Then, by \cite[Corollary 3.2.1 page 66]{kuvcera1991algebraic}, the sizes of the Jardan blocks associated to pure imahginaty eigenvalues of $H$ are all even.
%	Then, $\ker \left(H\right)$ contains a nonzero vector $v$. Let $v_1$ the projection of $v$ onto the first $n$ components of $\mathbb{R}^{2n}$ and $v_2$ the projection onto the second $n$ components. Set the initial datum $x_0\coloneqq v_1$ and the final target $v_1\coloneqq v_1$. Then, for any time horizon $T>0$, the solution to the optimality system $\left(x^T,p^T\right)\equiv v\neq 0$, whence both the optimal state and the optimal control are constant. By the strict convexity of \eqref{ocp_both_endpoints_fixed} and since $v\neq 0$, either the optimal control $u^T$ and the optimal state $x^T$ are constant and nonzero. Then, \eqref{ocp_both_endpoints_fixed_6} is not verified.
%	\\
	\textit{Step 1} \  \textbf{Hyperbolicity of the Hamiltonian matrix}\\
	If the weak Hautus test \eqref{weakHtest} is verified, by Lemma \ref{lemma_Ham_spectrum}, the critical subspace $\mathscr{L}^{0}(H)=\left\{0\right\}$.\\
	\textit{Step 2} \  \textbf{Uniqueness of the minimizer for the steady problem}\\
	By the above step, we have $\mathscr{L}^{0}(H)=\left\{0\right\}$. Then, the steady optimality system 
	\begin{equation}\label{steady_OS}
	\begin{dcases}
	A\overline{x}-BB^*\overline{p}&=0\\
	-A^*\overline{p}-C^*(C\overline{x}-z)&=0\\
	\end{dcases}
	\end{equation}
	admits a unique solution $\left(\overline{x},\overline{p}\right)$, whence the optimal pair for the minimization problem \eqref{steady_functional_linear_quadratic_fixed_endpoint} is unique and is given by $\left(\overline{u},\overline{x}\right)=\left(-B^*\overline{p},\overline{x}\right)$.\\
	\textit{Step 3} \  \textbf{Conclusion}\\
	If the weak Hautus test \eqref{weakHtest} is verified, by Lemma \ref{lemma_Ham_spectrum}, the critical subspace $\mathscr{L}^{0}(H)=\left\{0\right\}$. Then, Lemma \ref{lemma_infnorm_bound} and Lemma \ref{lemma_1} allow us to conclude.
\end{proof}

\subsection{Exponential velocity turnpike}
\label{subsec:2.2}

This subsection has been inspired by \cite{faulwasser2019towards} where the notion of velocity turnpike has been introduced. In particular, we give a theoretical explanation of the illustrative example in \cite[section 3]{faulwasser2019towards}. Note that, in the proposition below the Hamiltonian matrix may have imaginary eigenvalues.

By \cite[fact 1.(f)]{callier1995convergence},
\begin{equation}
	\mathscr{L}^0\left(A-BB^*\widehat{E}\right)=NO^0(C,A).
\end{equation}
Then, we can decompose
\begin{equation}
	\mathbb{R}^n = NO^0(C,A) \oplus \mathscr{L}^{-}\left(A-BB^*\widehat{E}\right)
\end{equation}
and define the corresponding projections $\mathbb{P}_1$ onto $NO^0(C,A)$ and $\mathbb{P}_2$ onto\\
$\mathscr{L}^{-}\left(A-BB^*\widehat{E}\right)$.

\begin{proposition}\label{prop_controllability_case_velocity}
	Suppose $(A,B)$ is controllable and $NO^0(C,A)\subseteq \ker\left(A\right)$. Take $T>2$. Let $u^T$ be an optimal control for \eqref{linear_ocp_2}-\eqref{functional_ocp_2} and $x^T$ be the optimal state. Then, there exist $T$-independent $K$ and $\mu >0$ such that
	\begin{equation}\label{turnpike_estimate_velocity_1}
	\left\|u^T(t)-\hat{u}^T\right\|+\left\|\mathbb{P}_2x^T(t)-\hat{x}^T\right\|\leq K\left[\exp\left(-\mu t \right)+\exp\left(-\mu\left(T-t\right)\right)\right],
	\end{equation}
	and
	\begin{equation}\label{turnpike_estimate_velocity_3}
	\left\|u^T(t)-\hat{u}^T\right\|+\left\|\mathbb{P}_1\left[x^T(t)-\left(x_0-tBB^*\hat{q}^T\right)\right]\right\|\leq K\left[\exp\left(-\mu t \right)+\exp\left(-\mu\left(T-t\right)\right)\right],
	\end{equation}
	where $\hat{q}^T\in\ker\left(\left(A-BB^*\widehat{E}\right)^*\right)$, $\hat{x}^T\in \mathscr{L}^{-}\left(A-BB^*\widehat{E}\right)$ and $\hat{u}^T=-BB^*\widehat{E}\hat{x}^T-BB^*\hat{q}^T$. Furthermore,
	\begin{equation}\label{turnpike_estimate_velocity_5}
	\mbox{dist}\left(\left(\hat{u}^T,\hat{x}^T\right),\mbox{argmin}\left[J_s\right] \right)^2 \leq \frac{K}{T},
	\end{equation}
	for some $K$ is independent of the time horizon $T$.
\end{proposition}
\begin{proof}[Proof of Proposition \ref{prop_controllability_case_velocity}]
	%\textit{Step 0} \ \textbf{Reduction to the case running target $z=0$}\\
	The optimal control for the time-evolution problem \eqref{linear_ocp_2}-\eqref{functional_ocp_2} reads as $u^T=-B^*p^T$, where
\begin{equation}\label{OS_controllability_3}
\begin{dcases}
\dot{x}^T(t)=Ax^T(t)-BB^*p^T(t)\hspace{1 cm}& t\in (0,T)\\
-\dot{p}^T(t)=A^*p^T(t)+C^*\left(Cx^T(t)-z\right)&t\in (0,T)\\
u^T(t)=-B^*p^T(t)&t\in (0,T)\\
x^T(0)=x_0\\
x^T(T)=x_1.
\end{dcases}
\end{equation}
By working in perturbation variables, as in the step 1 of the proof of Lemma \ref{lemma_infnorm_bound}, we can reduce to the case $z=0$.\\
	\textit{Step 1} \  \textbf{Change of variable}\\
	Set the linear transformation
	\begin{equation}\label{prop_controllability_case_velocity_defLambda}
	\Lambda \coloneqq 		\begin{bmatrix}
	I_n&0\\
	-\widehat{E}&I_n.
	\end{bmatrix}
	\end{equation}
	By using the Algebraic Riccati Equation \eqref{ARE}, we have\footnote{
		\begin{equation*}
		\Lambda^{-1}=\begin{bmatrix}
		I_n&0\\
		\widehat{E}&I_n.
		\end{bmatrix}
		\end{equation*}
	}
	\begin{equation*}
	\Lambda \hspace{0.03 cm} \mbox{Ham} \hspace{0.03 cm} \Lambda^{-1} = \begin{bmatrix}
	A-BB^*\widehat{E}&-BB^*\\
	0&-\left(A-BB^*\widehat{E}\right)^*.
	\end{bmatrix}
	\end{equation*}
	Set further $q^T\coloneqq -\widehat{E} x^T+p^T$. Then the pair $\left(x^T,q^T\right)$ solves
	\begin{equation}\label{OS_controllability_newvariables}
	\begin{dcases}
	\dot{x}^T(t)=\left(A-BB^*\widehat{E}\right)x^T(t)-BB^*q^T(t)\hspace{1 cm}& t\in (0,T)\\
	-\dot{q}^T(t)=\left(A-BB^*\widehat{E}\right)^*q^T(t)&t\in (0,T)\\
	x^T(0)=x_0\\
	x^T(T)=x_1.
	\end{dcases}
	\end{equation}
	\textit{Step 2} \  \textbf{Proof of the equality $\mathscr{L}^0\left(A-BB^*\widehat{E}\right)=\ker\left(A-BB^*\widehat{E}\right)$}\\
	By \cite[fact 1.(f)-(d)]{callier1995convergence} and the hypothesis $NO^0(C,A)\subseteq \ker(A)$, we have
	\begin{equation}\label{prop_controllability_case_velocity_eq24}
		\mathscr{L}^0\left(A-BB^*\widehat{E}\right)=NO^0(C,A)\subseteq \ker(A),
	\end{equation}
	and
	\begin{equation}\label{prop_controllability_case_velocity_eq26}
		\mathscr{L}^0\left(A-BB^*\widehat{E}\right)= NO^0(C,A)\subseteq \ker\left(\widehat{E}\right),
	\end{equation}
	whence
	\begin{equation}\label{prop_controllability_case_velocity_eq27}
	\mathscr{L}^0\left(A-BB^*\widehat{E}\right)\subseteq \ker\left(A\right)\cap\ker\left(\widehat{E}\right)\subseteq \ker\left(A-BB^*\widehat{E}\right),
	\end{equation}	
	Therefore, by definition of critical subspace $\mathscr{L}^0$
	\begin{equation}\label{prop_controllability_case_velocity_eq28}
		\mathscr{L}^0\left(A-BB^*\widehat{E}\right)=\ker\left(A-BB^*\widehat{E}\right).
	\end{equation}
	\textit{Step 3} \  \textbf{Estimate \eqref{turnpike_estimate_velocity_1}}\\
	From step 2, we have also
	%by Th(rank) applied to \left(A-BB^*\widehat{E}\right)^k, with k\in \mathbb{N}.
	\begin{equation}\label{prop_controllability_case_velocity_eq30}
		\mathscr{L}^0\left(\left(A-BB^*\widehat{E}\right)^*\right)=\ker\left(\left(A-BB^*\widehat{E}\right)^*\right).
	\end{equation}
	To finish the proof, we firstly focus on the second equation in \eqref{OS_controllability_newvariables}, satisfied by $q^T$. By \eqref{prop_controllability_case_velocity_eq30}, there exists $\hat{q}^T\in \mathscr{L}^0\left(\left(A-BB^*\widehat{E}\right)^*\right)=\ker\left(\left(A-BB^*\widehat{E}\right)^*\right)$, such that
	\begin{equation}\label{prop_controllability_case_velocity_eq32}
		\left\|q^T(t)-\hat{q}^T\right\|\leq K\left[\exp\left(-\mu (T-t)\right)\right]\left\|q^T(T)\right\|,\hspace{0.3 cm}\forall \ t\in [0,T]
	\end{equation}
	the constant $K$ and $\mu>0$ being independent of $T>2$. We need now to estimate $\left\|q^T(T)\right\|$, uniformly on $T>2$. By definition of the linear transformation $\Lambda$, we have
	\begin{equation}
		q^T(T) = p^T(T)-\widehat{E}x^T(T) = p^T(T)-\widehat{E}x_1.
	\end{equation}
	By Lemma \ref{lemma_infnorm_bound}, there exists $K=K(A,B,C)$ such that, for any time horizon $T>2$ and time instant $t\in [0,T]$, we have
	\begin{equation}
	\left\|p^T(t)\right\|\leq K\left[\left\|x_0\right\|+\left\|x_1\right\|+\left\|z\right\|\right],
	\end{equation}
	whence
\begin{equation}\label{prop_controllability_case_velocity_eq34}
\left\|	q^T(T)\right\|\leq \left\|p^T(T)\right\|+K\left\|x_1\right\|\leq K\left[\left\|x_0\right\|+\left\|x_1\right\|+\left\|z\right\|\right]
\end{equation}
	Therefore, by \eqref{prop_controllability_case_velocity_eq32} and \eqref{prop_controllability_case_velocity_eq34}, we have
	\begin{equation}\label{prop_controllability_case_velocity_eq36}
	\left\|q^T(t)-\hat{q}^T\right\|\leq K\left[\exp\left(-\mu (T-t)\right)\right],\hspace{0.3 cm}\forall \ t\in [0,T]
	\end{equation}
	the constants $\mu=\mu(A,B,C)$ and $K=K(A,B,C,x_0,x_1,z)$.
	
	At this stage, by definition of $\mathbb{P}_2$,
	\begin{equation}\label{linear_ocp_6}
	\begin{dcases}
	\frac{d}{dt}\left[\mathbb{P}_2x^T\right]=\mathbb{P}_2\left(A-BB^*\widehat{E}\right)\mathbb{P}_2x^T-\mathbb{P}_2BB^*q^T\hspace{2.8 cm} & \mbox{in} \hspace{0.10 cm}(0,T)\\
	\mathbb{P}_2x^T(0)=\mathbb{P}_2x_0.
	\end{dcases}
	\end{equation}
	Now, $\mathbb{P}_2\left(A-BB^*\widehat{E}\right)\mathbb{P}_2$ is stable, which, together with \eqref{prop_controllability_case_velocity_eq36}, leads to
	\begin{equation}\label{prop_controllability_case_velocity_eq39}
		\left\|\mathbb{P}_2x^T(t)-\hat{x}^T\right\|\leq K\left[\exp\left(-\mu t \right)+\exp\left(-\mu\left(T-t\right)\right)\right],
	\end{equation}
	for some $\hat{x}^T\in \mathscr{L}^{-}\left(A-BB^*\widehat{E}\right)$. Now, by \eqref{prop_controllability_case_velocity_eq39} and \eqref{prop_controllability_case_velocity_eq26}, the optimal control
	\begin{equation}
		u^T=-B^*\widehat{E}x^T-B^*q^T=-B^*\widehat{E}\mathbb{P}_2x^T-B^*q^T.
	\end{equation}
	Then, for  $\hat{u}^T=-BB^*\widehat{E}\hat{x}^T-BB^*\hat{q}^T$, we have
	\begin{equation}\label{prop_controllability_case_velocity_eq42}
	\left\|u^T(t)-\hat{u}^T\right\|\leq K\left[\exp\left(-\mu t \right)+\exp\left(-\mu\left(T-t\right)\right)\right].
	\end{equation}
	This finishes the proof of \eqref{turnpike_estimate_velocity_1}.\\
	\textit{Step 4} \  \textbf{Proof of \eqref{turnpike_estimate_velocity_3}}\\
	To conclude, it suffices to prove
	\begin{equation}\label{turnpike_estimate_velocity_6}
	\left\|\mathbb{P}_1\left[x^T(t)-\left(x_0-tBB^*\hat{q}^T\right)\right]\right\|\leq K\left[\exp\left(-\mu t \right)+\exp\left(-\mu\left(T-t\right)\right)\right].
	\end{equation}
	We have
	\begin{equation}\label{linear_ocp_12}
	\begin{dcases}
	\frac{d}{dt}\left[\mathbb{P}_1x^T\right]=\mathbb{P}_1\left(A-BB^*\widehat{E}\right)\mathbb{P}_1x^T-\mathbb{P}_1BB^*q^T\hspace{2.8 cm} & \mbox{in} \hspace{0.10 cm}(0,T)\\
	\mathbb{P}_1x^T(0)=\mathbb{P}_1x_0.
	\end{dcases}
	\end{equation}
	Now, by assumption, the kernel $\ker\left(A-BB^*\widehat{E}\right)=\mathscr{L}^0\left(A-BB^*\widehat{E}\right)=NO^0(C,A)$, whence
	\begin{equation}\label{linear_ocp_16}
	\begin{dcases}
	\frac{d}{dt}\left[\mathbb{P}_1x^T\right]=-\mathbb{P}_1BB^*q^T\hspace{2.8 cm} & \mbox{in} \hspace{0.10 cm}(0,T)\\
	\mathbb{P}_1x^T(0)=\mathbb{P}_1x_0,
	\end{dcases}
	\end{equation}
	i.e. $\mathbb{P}_1x^T$ reads as
	\begin{eqnarray}
	\mathbb{P}_1x^T(t)&=&\mathbb{P}_1x_0-\int_0^t\mathbb{P}_1BB^*q^T(s)ds\nonumber\\
	&=&\mathbb{P}_1x_0-\int_0^t\mathbb{P}_1BB^*\hat{q}^Tds-\int_0^t\mathbb{P}_1BB^*\left[q^T(s)-\hat{q}^T\right]ds\nonumber\\
	&=&\mathbb{P}_1x_0-t\mathbb{P}_1BB^*\hat{q}^T-\int_0^t\mathbb{P}_1BB^*\left[q^T(s)-\hat{q}^T\right]ds.
	\end{eqnarray}
	By \eqref{prop_controllability_case_velocity_eq36}, we have, for any $t\in [0,T]$
	\begin{eqnarray}
	\int_0^t\left\|\mathbb{P}_1BB^*\left[q^T(s)-\hat{q}^T\right]\right\|ds&\leq&K\int_0^t\exp\left(-\mu (T-s)\right)ds\nonumber\\
	&=&\frac{K}{\mu}\left[\exp\left(-\mu (T-t)\right)-\exp\left(-\mu T\right)\right]\nonumber\\
	&\leq&\frac{K}{\mu}\exp\left(-\mu (T-t)\right),
	\end{eqnarray}
	whence
	\begin{equation}
		\left\|\mathbb{P}_1x^T(t)-\mathbb{P}_1x_0-t\mathbb{P}_1BB^*\hat{q}^T\right\|\leq \frac{K}{\mu}\exp\left(-\mu (T-t)\right),
	\end{equation}
	as desired.\\
	\textit{Step 5} \  \textbf{Proof of \eqref{turnpike_estimate_velocity_5}}\\
	On the one hand, defining a control $\hat{u}$ as in \eqref{picewisecontrol}, we have the upper bound
	\begin{equation}\label{est_min_val_functional_turnpike_estimate_velocity}
	J^T(u^T)\leq J^T(\hat{u})\leq K\left[\|x_0\|^2+\|x_1\|^2\right].
	\end{equation}
	On the other hand, employing \eqref{turnpike_estimate_velocity_1} we get the lower bound
	\begin{equation}\label{turnpike_estimate_velocity_12}
	J^T(u^T)\geq T\left[\left\|\hat{u}^T\right\|^2+\left\|C\hat{x}^T\right\|^2\right]-K,
	\end{equation}
	where the constant $K$ is independent of the time horizon $T$ and the terminal data $x_0$ and $x_1$. Since $\hat{x}^T\in \mathscr{L}^{-}\left(A-BB^*\widehat{E}\right)$,
	%employ Remark(07/04/2020_03)
	 \eqref{est_min_val_functional_turnpike_estimate_velocity} together with \eqref{turnpike_estimate_velocity_12} yields \eqref{turnpike_estimate_velocity_5}. This finishes the proof.
\end{proof}

\begin{example}
	We show now how our techniques apply in the example presented in \cite[section 3]{faulwasser2019towards}. We consider the control system
	\begin{equation}\label{linear_ocp_2_example}
	\begin{dcases}
	\dot{x}_1=x_2\hspace{2.8 cm} & \mbox{in} \hspace{0.10 cm}(0,T)\\
	\dot{x}_2=u\hspace{2.8 cm} & \mbox{in} \hspace{0.10 cm}(0,T)\\
	x_1(0)=x_{0,1}, \ x_2(0)=x_{0,2}, \ x_1(T)=x_{1,1}, \ x_2(T)=x_{1,2}.
	\end{dcases}
	\end{equation}
	
	We introduce the set of admissible controls
	\begin{equation}\label{ocp_both_endpoints_fixed_example}
	\mathscr{U}_{\mbox{\tiny{ad}}}\coloneqq \left\{u\in L^2(0,T;\mathbb{R}) \ | \ \mbox{there exists a solution}\hspace{0.3 cm} x \hspace{0.3 cm} \mbox{to \eqref{linear_ocp_2_example}}\right\}.
	\end{equation}
	
	We formulate the optimal control problem
	\begin{equation}\label{functional_ocp_2_example}
	\min_{u\in \mathscr{U}_{\mbox{\tiny{ad}}}}J^{T}(u)=\frac12 \int_{0}^T \left[|u(t)|^2+|x_2(t)|^2\right] dt.
	\end{equation}
	For this special problem, we have $NO^0(C,A)= \ker\left(A\right)$, whence Proposition \ref{prop_controllability_case_velocity} is applicable.
	
	We write the Algebraic Riccati Equation associated to the above problem
	\begin{equation}\label{ARE_example}
	\begin{bmatrix}
	0&0\\
	1&0
	\end{bmatrix}\widehat{E}+\widehat{E}\begin{bmatrix}
	0&1\\
	0&0
	\end{bmatrix}-\widehat{E}\begin{bmatrix}
	0&1\\
	0&0
	\end{bmatrix}\widehat{E}+\begin{bmatrix}
0&0\\
0&1
\end{bmatrix}=0.
	\end{equation}
	
	The unique positive semidefinite solution to the above equation is given by
	\begin{equation}
		\widehat{E}=\begin{bmatrix}
	0&0\\
0&1
\end{bmatrix},
	\end{equation}
	whence
	\begin{equation}
		A-BB^*\widehat{E}=\begin{bmatrix}
	0&1\\
0&-1
\end{bmatrix},
	\end{equation}
	with spectrum $\sigma\left(A-BB^*\widehat{E}\right)=\left\{0,-1\right\}$. We have
	\begin{itemize}
		\item $\ker\left(A-BB^*\widehat{E}\right)=\mbox{span}\left\{\mathbf{e}_1\right\}$;
		\item $\ker\left(\left(A-BB^*\widehat{E}\right)^*\right)=\left\{\left(x_1,x_1\right) \ | \ x_1\in \mathbb{R} \right\}$;
		\item $\mathscr{L}^{-}\left(A-BB^*\widehat{E}\right)=\left\{\left(x_1,-x_1\right) \ | \ x_1\in \mathbb{R} \right\}$;
		\item $\mathscr{L}^{-}\left(\left(A-BB^*\widehat{E}\right)^*\right)=\mbox{span}\left\{\mathbf{e}_2\right\}$.
	\end{itemize}
	We decompose
	\begin{equation}
	\mathbb{R}^n = NO^0(C,A) \oplus \mathscr{L}^{-}\left(A-BB^*\widehat{E}\right)=\mbox{span}\left\{\mathbf{e}_1\right\}\oplus \left\{\left(x_1,-x_1\right) \ | \ x_1\in \mathbb{R} \right\},
	\end{equation}
	with the corresponding projections $\mathbb{P}_1(x_1,x_2)=\left(x_1+x_2,0\right)$ and $\mathbb{P}_2(x_1,x_2)=\left(-x_2,x_2\right)$.
	
	By Proposition \ref{prop_controllability_case_velocity}, there exist $T$-independent $K$ and $\mu >0$ such that
	\begin{equation}\label{turnpike_estimate_velocity_1_example}
	\left|u^T(t)-\hat{u}^T\right|+\left\|\mathbb{P}_2x^T(t)-\hat{x}^T\right\|\leq K\left[\exp\left(-\mu t \right)+\exp\left(-\mu\left(T-t\right)\right)\right],
	\end{equation}
	and
	\begin{equation}\label{turnpike_estimate_velocity_3_example}
	\left|u^T(t)-\hat{u}^T\right|+\left\|\mathbb{P}_1\left[x^T_1(t)-\left(x_{0,1}-tBB^*\hat{q}^T\right)\right]\right\|\leq K\left[\exp\left(-\mu t \right)+\exp\left(-\mu\left(T-t\right)\right)\right],
	\end{equation}
	where $\hat{q}^T\in\ker\left(\left(A-BB^*\widehat{E}\right)^*\right)=\left\{\left(x_1,x_1\right) \ | \ x_1\in \mathbb{R} \right\}$,\\
	$\hat{x}^T\in \mathscr{L}^{-}\left(A-BB^*\widehat{E}\right)=\left\{\left(x_1,-x_1\right) \ | \ x_1\in \mathbb{R} \right\}$ and $\hat{u}^T=-BB^*\widehat{E}\hat{x}^T-BB^*\hat{q}^T$. Moreover,
	\begin{equation}\label{turnpike_estimate_velocity_5_example}
	\left|\hat{u}^T\right|^2+\left\|\hat{x}^T_2\right\|^2=\mbox{dist}\left(\left(\hat{u}^T,\hat{x}^T\right),\mbox{argmin}\left[J_s\right] \right)^2 \leq \frac{K}{T},
	\end{equation}
	where $K$ is independent of the time horizon $T$.
\end{example}

\appendix

\section{Proof of Lemma \ref{lemma_minimizer steady}.}
\label{sec:Proof of Lemma lemma minimizer steady.}

We determine the set of minimizers of \eqref{Js}, to prove Lemma \ref{lemma_minimizer steady}.
\begin{proof}[Proof of Lemma \ref{lemma_minimizer steady}.]
	\textit{Step 1} \ \textbf{Existence of minimizer}\\
	We introduce the equivalence relation in $M$
	\begin{equation*}
	\sim: \hspace{0.3 cm}(x_1,u_1)\sim (x_2,u_2)\hspace{0.6 cm}\mbox{if}\hspace{0.6 cm} v_1=v_2\hspace{0.3 cm}\mbox{and}\hspace{0.3 cm}y_2-y_1\in \ker(A)\cap\ker(C).
	\end{equation*}
	We denote by $[(x,u)]$ the equivalence class of $(x,u)$. Actually, $J_s(x_1,u_1)=J_s(x_2,u_2)$, provided that $(x_1,u_1)\sim (x_2,u_2)$. Then, we are in position to define
	\begin{equation*}
	\left[J_s\right]:M/\sim\longrightarrow\mathbb{R}
	\end{equation*}
	\begin{equation*}
	[(x,u)]\longmapsto J_s(x,u).
	\end{equation*}
	Now, by definition of $J_s$ and ${M}/{\sim}$, for any $r \geq 0$ the sublevel set
	\begin{equation*}
	\mathscr{S}_r\coloneqq \left\{[(x,u)]\in {M}/{\sim} \ | \ J_s(x,u)\leq r\right\}
	\end{equation*}
	is compact. Hence, by the Weierstrass extreme value theorem, there exists global minimizer $[(\overline{u},\overline{x})]$ for $\left[J_s\right]$. Then, $(\overline{u},\overline{x})$ is a global minimizer of $J_s$.\\
	\textit{Step 2} \ \textbf{Conclusion}\\
	By definition of $(\overline{u},\overline{x})$ and $J_s$,
	\begin{equation*}
	\mbox{argmin}(J_s)\supseteq\left\{(\overline{u},\overline{x})\right\}+\left\{0\right\}\times\left[\ker(A)\cap \ker(C)\right].
	\end{equation*}
	Let us now prove the other inclusion by contradiction. Suppose there exists $(x,u)$ global minimizer such that
	\begin{equation*}
	(x,u)\notin \left\{(\overline{u},\overline{x})\right\}+\left\{0\right\}\times\left[\ker(A)\cap \ker(C)\right].
	\end{equation*}
	Then, either $v\neq \overline{u}$ or
	\begin{equation*}
	v=\overline{u}\hspace{0.6 cm}\mbox{and}\hspace{0.6 cm} y-\overline{x}\notin \ker(A)\cap \ker(C).
	\end{equation*}
	In both cases $(v,Cy)\neq (\overline{u},C\overline{x})$. Indeed, in the first case, from $v\neq \overline{u}$, we have $(v,Cy)\neq (\overline{u},C\overline{x})$. In the second case $v=\overline{u}$. Therefore, $A(y-\overline{x})=B(v-\overline{u})=0$. Then, $y-\overline{x}\in \ker(A)$, whence $y-\overline{x}\notin \ker(C)$. Now, the function
	\begin{equation*}
	g: \mathbb{R}^m\times \mathbb{R}^n\longmapsto \mathbb{R}
	\end{equation*}
	\begin{equation*}
	(v,\tilde{y}) \longrightarrow \frac12\left[\|v\|^2+\|\tilde{y}-z\|^2\right]
	\end{equation*}
	is strictly convex.
	Then, $J_s(x,u)=g(v,Cy)\neq g(\overline{u},C\overline{x})=J_s(\overline{u},\overline{x})$. Hence $(x,u)$ is not a minimizer, so obtaining a contradiction. This finishes the proof.
\end{proof}

\section{Kernel and range of the Hamiltonian matrix}
\label{appendixsec:Kernel and range of the Hamiltonian matrix}

Let $A\in\mathcal{M}_{n\times n}(\mathbb{R})$, $B\in\mathcal{M}_{n\times m}(\mathbb{R})$ and $C\in\mathcal{M}_{n\times n}(\mathbb{R})$. We determine the kernel and the range of the Hamiltonian matrix
\begin{equation}\label{Ham_Appendix}
\mbox{Ham}\coloneqq \begin{bmatrix}
A&-BB^*\\
-C^*C&-A^*.
\end{bmatrix}
\end{equation}
The above is the coefficient matrix of the optimality system \eqref{steady_OS} for the steady problem \eqref{steady_functional_linear_quadratic_finite_dimension}.

\begin{lemma}\label{lemma_Hamkerrange}
	Let $\mbox{Ham}$ be the Hamiltonian matrix \eqref{Ham_Appendix}. Then,
	\begin{equation}
		\ker\left(\mbox{Ham}\right)=\left[\ker\left(A\right)\cap \ker\left(C\right)\right]\times \left[\ker\left(A^*\right)\cap \ker\left(B^*\right)\right]
	\end{equation}
	and
	\begin{equation}
		\mbox{Range}\left(\mbox{Ham}\right)=\left[\mbox{Range}\left(A\right) + \mbox{Range}\left(B\right)\right]\times \left[\mbox{Range}\left(A^*\right)+ \mbox{Range}\left(C^*\right)\right].
	\end{equation}
\end{lemma}
\begin{proof}[Proof of Lemma \ref{lemma_Hamkerrange}]
	\textit{Step 1} \  \textbf{Computation of the kernel}\\
	By definition we have the inclusion
	\begin{equation}
		\left[\ker\left(A\right)\cap \ker\left(C\right)\right]\times \left[\ker\left(A^*\right)\cap \ker\left(B^*\right)\right]\subseteq \ker\left(\mbox{Ham}\right).
	\end{equation}
	We are now going to prove the other inclusion. Take an arbitrary $\left(\overline{x},\overline{p}\right)\in \ker \left(\mbox{Ham}\right)$, which verifies
		\begin{equation}\label{steady_OS_zero}
	\begin{dcases}
	A\overline{x}-BB^*\overline{p}&=0\\
	-A^*\overline{p}-C^*C\overline{x}&=0\\
	\end{dcases}
	\end{equation}
	We start by multiplying the first equation by $\overline{p}$ and the second equation by $\overline{x}$, getting
	\begin{equation}
		0=\left(A\overline{x},\overline{p}\right)-\left(BB^*\overline{p},\overline{p}\right)=\left(\overline{x},A^*\overline{p}\right)-\left\|B^*\overline{p}\right\|^2
	\end{equation}
	and
	\begin{equation}
		0=-\left(\overline{x},A^*\overline{p}\right)-\left(C^*C\overline{x},\overline{x}\right)=-\left(\overline{x},A^*\overline{p}\right)-\left\|C^*\overline{x}\right\|^2.
	\end{equation}
	We sum the above equation, obtaining
	\begin{equation}
		\left\|B^*\overline{p}\right\|^2+\left\|C\overline{x}\right\|^2=0,
	\end{equation}
	whence $\overline{p}\in \ker\left(B^*\right)$ and $\overline{x}\in \ker\left(C\right)$, which, together with \eqref{steady_OS_zero}, leads to
	\begin{equation}\label{steady_OS_zero_3}
	\begin{dcases}
	A\overline{x}&=BB^*\overline{p}=0\\
	A^*\overline{p}&=-C^*C\overline{x}=0.\\
	\end{dcases}
	\end{equation}
	This finishes step 1.\\
	\textit{Step 2} \  \textbf{Computation of the range}\\
	By linear algebra,
	\begin{equation}
		\mbox{Range}\left(\mbox{Ham}\right)=\ker\left(\mbox{Ham}^*\right)^{\perp}.
	\end{equation}
	Now, the transpose of $\mbox{Ham}$ reads as
	\begin{equation}
		\mbox{Ham}^*=\begin{bmatrix}
		A^*&-C^*C\\
		-BB^*&-A.
		\end{bmatrix}
	\end{equation}
	Set $\tilde{A}\coloneqq A^*$, $\tilde{B}\coloneqq C^*$ and $\tilde{C}\coloneqq B^*$. The above matrix is the hamiltonian matrix for $\tilde{A}$, $\tilde{B}$ and $\tilde{C}$. Then, by using the results of step 1, we have
	\begin{equation}
		\ker\left(\mbox{Ham}^*\right)=\left[\ker\left(A^*\right)\cap \ker\left(B^*\right)\right]\times \left[\ker\left(A\right)\cap \ker\left(C\right)\right],
	\end{equation}
	whence
	\begin{equation}
		\ker\left(\mbox{Ham}^*\right)^{\perp}=\left[\mbox{Range}\left(A\right)+ \mbox{Range}\left(B\right)\right]\times \left[\mbox{Range}\left(A^*\right)+ \mbox{Range}\left(C^*\right)\right],
	\end{equation}
	as desired.
\end{proof}

\section{Proof of Lemma \ref{lemma_1}}
\label{appendixsec:Proof of Lemma lemma 1}

\begin{proof}[Proof of Lemma \ref{lemma_1}]
	\textit{Step 1} \  \textbf{Stable, antistable and critical splitting}\\
	We have
	\begin{equation*}
	\mathbb{R}^n=\mathscr{L}^{-}(H)\oplus \mathscr{L}^{0}(H)\oplus\mathscr{L}^{+}(H),
	\end{equation*}
	where $\oplus$ stands for the direct sum. Then, let $y$ be a solution to \eqref{linear_nohomo}. Denote by $y_1$, $y_2$ and $y_3$ resp. the projections of $x$ onto $\mathscr{L}^{-}(H)$, $\mathscr{L}^{0}(H)$ and $\mathscr{L}^{+}(H)$. Then, $x=y_1+y_2+y_3$ and, for $i=1,2,3$,
	\begin{equation*}
	\dot{y}_i=Hy_i\hspace{2.8 cm}  \mbox{in} \hspace{0.10 cm}(0,T).
	\end{equation*}
	\textit{Step 2} \ \textbf{Estimate for the stable part}\\
	We have
	\begin{equation*}
	\dot{y}_1=Hy_1\hspace{2.8 cm}  \mbox{in} \hspace{0.10 cm}(0,T),
	\end{equation*}
	All the eigenvalues of $H\hspace{-0.1 cm}\restriction_{\mathscr{L}^{-}(H)}$ have strictly negative real part, where we have denoted by $L_H$ the linear operator associated to the matrix $H$. Then,
	we have, for any $t\in [0,T]$
	\begin{equation}\label{estimate_stable_pw}
	\|y_1(t)\|\leq K\exp(-\mu t)\|y_1(0)\|\leq K\exp(-\mu t)\|y(0)\|,
	\end{equation}
	the constant $K$ depending only on $H$.\\
	\textit{Step 3} \ \textbf{Estimate for the unstable part}\\
	By definition
	\begin{equation}\label{unstable_part}
	\dot{y}_3=Hy_3\hspace{2.8 cm}  \mbox{in} \hspace{0.10 cm}(0,T).
	\end{equation}
	Then, $\tilde{y}_3(t)\coloneqq y_3(T-t)$ solves
	\begin{equation}\label{reversed_unstable_part}
	\dot{\tilde{y}}_3=-H\tilde{y}_3\hspace{2.8 cm}  \mbox{in} \hspace{0.10 cm}(0,T).
	\end{equation}
	Now, all the eigenvalues of $-H\hspace{-0.1 cm}\restriction_{\mathscr{L}^{+}(H)}$ have strictly negative real part. Then, as in Step 1,
	\begin{equation}\label{estimate_unstable_pw}
	\|y_3(t)\|=\|\tilde{y}_3(T-t)\|\leq K\exp(-\mu (T-t))\|\tilde{y}_3(0)\|\leq K\exp(-\mu (T-t))\|y(T)\|.
	\end{equation}
	The estimates 
	%together with Lemma [CHECK]
	\eqref{estimate_stable_pw} and \eqref{estimate_unstable_pw} yield \eqref{est_L_inf}.
	%see Remark (23/10/2019_08).
\end{proof}

\bibliography{my_references}
\bibliographystyle{siam}

\medskip
% The data information below will be filled by AIMS editorial staff
Received xxxx 20xx; revised xxxx 20xx.
\medskip

\end{document}